\documentclass[11pt]{amsart}
\usepackage{quick-notes}

\usepackage{hyperref}
\hypersetup{%
colorlinks=true,
linkcolor=blue,
citecolor=red
}%

\title{The Levelwise Finite Generation of Free
Tambara Functors}
\author{Emory Sun}
\date{}
\address{Columbia University}
\email{es4058@columbia.edu}

\begin{document}

\begin{abstract}
We prove the levelwise finite generation
of free polynomial $G$-Tambara functors when $G$ is a finite
Dedekind group.
In the process, we establish the permanence of various
finiteness conditions under box products and norms $n_H^G$
of Tambara functors, including
a weak Hilbert Basis Theorem.
\end{abstract}

\maketitle

\tableofcontents

\section{Introduction}

\subsection{Background}

Tambara functors are equivariant
analogues of commutative rings, appearing naturally
as ring structures associated to systems of representations.
Representation rings, Galois extensions,
Burnside rings, and even commutative rings with a $G$-action
in the simplest case all give rise to Tambara functors,
which consist of a collection of commutative rings
(known as the \textit{levels} of the Tambara functor)
with structure maps between them.

The notion of
a \textit{free polynomial Tambara functor,} first introduced
in \cite{blumberg2018incomplete}, is the equivariant
analogue of a free polynomial ring; they represent
the functors which send a Tambara functor to
its underlying set in a chosen level.
In some sense, these can be viewed as very large
Burnside rings---indeed, specializing
to the generator associated to the $G$-set $X = \varnothing$
recovers the usual Burnside ring.

Free polynomial Tambara functors both mirror and differ from
their classical counterparts in interesting ways.
A result of \cite{BRUN2005233} shows that
the bottom level of any free polynomial Tambara
functor $\A[X]$ is a free polynomial ring with
$\abs{X}$ generators.
On the other hand, in \cite{Hill_2022},
free Tambara functors are shown to be almost never
flat, a striking deviation from classical algebra.

However, not much is known about the levelwise
structure of polynomial Tambara functors aside from
a collection of specialized cases.
In \cite{schuchardt2025algebraicallyclosedfieldsequivariant},
the free polynomial Tambara functor on a top-level generator
is shown to be levelwise finitely generated.
A weak Hilbert basis theorem for free
polynomial Tambara functors was shown for $G = C_p$
by recent work in \cite{chan2025tambaraaffineline},
where a computation of their Nakaoka spectra is also
given.
Knowing such results in greater generality
allows us to better apply the wealth of knowledge from
classical commutative algebra in the equivariant setting.

\subsection{Main Results}

In this paper, we establish several finiteness
results for free polynomial Tambara functors.
First, a definition:

\begin{introdefn*}[Definition \ref{def:tambara_finite}]
A finite group $G$ is said to be
\textit{Tambara finite}
if the free polynomial
Tambara functor $\A[X]$ is levelwise
finitely generated for all finite $G$-sets $X$,
i.e. $\A[X](G/L)$ is a finite-type $\Z$-algebra
for all $L \leq G$. We say that $G$ is
\textit{transitively Tambara finite} or just
\textit{transitively finite} if $\A[X]$ is
levelwise finitely generated for all
\textit{transitive} finite $G$-sets $X$.
\end{introdefn*}

In \S\ref{sec:boxproduct}, we show that
being Tambara finite is equivalent to being
transitively Tambara finite.
The Tambara-finiteness of $G = C_p$ is established
in \cite{chan2025tambaraaffineline};
our main theorem is the extension of this result
to a much larger class of groups. More precisely, we prove:

\begin{introthm}[Theorem \ref{thm:finite_generation_general}]
Let $G$ be a finite Dedekind group, i.e. every subgroup
of $G$ is normal.
Then $G$ is Tambara finite.
\end{introthm}

In particular, the class of groups $G$
to which the theorem applies includes all finite abelian
groups and the quaternion group $Q_8$.

This is a fairly strong finiteness condition on $\A[X]$.
For general $G$, we prove a weaker theorem:

\begin{introthm}[Theorem \ref{thm:relatively_finite_general}]
Let $G$ be a finite group and $X$ be a finite $G$-set.
Then $\A[X]$ is relatively finite-dimensional,
in the sense that all restriction
maps $\Res_H^K$ are finite ring maps.
In particular, $\A[X]$ is module-Noetherian
(in the sense that
every submodule of a finitely generated module over $\A[X]$
is finitely generated)
when $G$ is Tambara finite.
\end{introthm}

The notion of a relatively finite-dimensional
Green or Tambara functor is new, due to
\cite{chanwisdom2025algebraicktheory},
and it is critical to extending any results obtained
for transitive $G$-sets to general $G$-sets,
as it is well-behaved under box products;
we will see this in \S\ref{sec:boxproduct}.
We also prove a weak Hilbert Basis Theorem
which generalizes the corresponding result for $G = C_p$
presented in \cite{chan2025tambaraaffineline}.


\begin{introthm}[Weak Hilbert Basis Theorem,
Theorem \ref{thm:hilbert_basis}]
Let $G$ be a Tambara finite group,
$X$ a finite $G$-set,
and $T$ a
levelwise Noetherian,
relatively finite-dimensional Tambara functor.
Then $T[X] = T \boxtimes \A[X]$ is also levelwise
Noetherian and relatively finite-dimensional.
If $T$ is levelwise finitely generated, then so is
$T[X]$.
\end{introthm}

In particular, under the assumptions above,
$T[X]$ is Noetherian, i.e. satisfies the ascending
chain condition on Tambara ideals.

Our proof of Theorem A proceeds by first establishing
transitive Tambara finiteness;
technical results on the box product of two Tambara functors
are developed in \S\ref{sec:boxproduct}
to show that transitively Tambara finite groups are also
Tambara finite. In particular,
a simple (noninductive) formula for the box product
of an arbitrary number of
$G$-Mackey functors for a finite group
$G$ is developed in \S\ref{sec:boxproduct}.

\begin{introthm}[Theorem \ref{thm:boxproduct}]
Let $G$ be a finite group,
$M_1, \ldots, M_N$ be a collection of $G$-Mackey functors
and write $M = M_1 \boxtimes \cdots \boxtimes M_N$. Fix a subgroup
$L \leq G$. For any subgroup $H \leq G$, define
\[
S_H = M_1(G/H) \otimes \cdots \otimes M_N(G/H).
\]
Then we have
\[
M(G/L) = \left( \bigoplus_{H \leq L} S_H \right)/F,
\]
where $F$ is the submodule generated by Frobenius and Weyl
relations, defined in \S\ref{sec:boxproduct}.
\end{introthm}

This generalizes previously known formulas for the box
product appearing in the literature for when $G$
is a cyclic $p$-group;
for instance, an inductive formula when
$G = C_{p^n}$ is described in \cite{MR3187614}.
The formula also has the advantage of having
an easy-to-describe ring structure when the $M_i$
are Tambara functors.
We use this to prove the following:

\begin{introthm}
Let $G$ be a finite group and
let $T, R$ be two levelwise finitely generated Tambara
functors. Then $T \boxtimes R$ is levelwise finitely generated.
\end{introthm}

This is a surprisingly delicate result, as it fails for Green
functors---a simple counterexample when $G = C_p$ is given in
\S\ref{sec:boxproduct}.



(Bi)incomplete Tambara functors are not considered in this paper,
which constitutes a direction for further investigation. While
we suspect that Theorem A holds for all
free polynomial Tambara functors for a finite group $G$,
it fails for Green functors in general
(in the sense that not all free Green functors for a finite group
$G$ are levelwise finitely generated).
In light of this fact, one might ask:
what conditions must one impose on
the indexing system of an incomplete Tambara
functor to ensure levelwise finite generation?



We note moreover
that a better understanding of the levelwise ring
structure of the free polynomial functors $\A[X]$
can be achieved if more explicit information were
known about the ring structure of the
norms $n_H^G$ of Tambara functors,
$H \leq G$. The norm functor $n_H^G$ is the left adjoint
to the restriction functor $r_H^G$ which sends
a $G$-Tambara functor to an $H$-Tambara functor via
precomposition by induction of $H$-sets---see
\cite{hill2019equivariant}
and \cite{rolfthesis}. In fact, our results show:

\begin{introthm}[Corollary \ref{cor:norm_preservation_general},
Theorem \ref{thm:norm_preservation}]
The norm functors $n_H^G$ preserve levelwise
finite generation over $\Z$ for all
finite $H \leq G$ iff all finite groups are
Tambara finite.
If
$H \leq G$ and $G$ is
a finite Dedekind group,
then $n_H^G$ preserves levelwise finite generation.
\end{introthm}


This leads to the natural followup
question:

\begin{question}
Let $T$ be an $H$-Tambara functor which is
levelwise Noetherian/relatively
finite-dimensional. Is the same true for
$n_H^G T$?
\end{question}

\subsection{Acknowledgements}

This work was completed as part of the 2025
Mathematics REU program at the University of Chicago.
I am deeply indebted to my REU mentor Noah Wisdom, for without
his guidance this project would not have been possible.
I would also like to thank David Chan,
Danika Van Niel, and David Mehrle for
their patience and willingness
to explain their work to me. I am incredibly grateful
to Peter May for his comments,
from which the organization of the paper has benefitted immensely.
Furthermore, I would like
to thank Michael A. Hill for his numerous suggested revisions
and for identifying many important errors.

\section{Review of Mackey and Tambara Functors}

\subsection{The Polynomial Category}

We briefly review the notions of equivariant
algebra we will need and establish some conventions
for the rest of this text. For details, we refer the reader
to \cite{strickland2012tambarafunctors}.
Fix a finite group $G$. We use $\bispan_G$ to denote the
\textit{category of bispans of finite $G$-sets} or
\textit{category of polynomials of finite $G$-sets,} where
objects are finite $G$-sets and morphisms are isomorphism classes
of \textit{polynomials} $[X \xleftarrow{p} A \xrightarrow{q} B
\xrightarrow{r} Y]$, and $\bispan_G^+$ the
\textit{category of spans of finite $G$-sets,} the subcategory
of $\bispan_G$ containing all the objects and where morphisms
are polynomials above such that $q$ is an isomorphism.
Here, an isomorphism of polynomials is described by a commutative
diagram of the form
\[
\begin{tikzcd}
& A \ar[dl] \ar[dd, "\cong"]\ar[r] & B \ar[dd, "\cong"] \ar[dr] & \\
X & & & Y \\
& A' \ar[ul] \ar[r] & B' \ar[ur] &
\end{tikzcd}
\]
A composition
of polynomials $[X \xleftarrow{p} A \xrightarrow{q} B
\xrightarrow{r} Y]$ and $[Y \xleftarrow{p'} A' \xrightarrow{q'}
B' \xrightarrow{r'} Z]$
is given by $[X \ot A'' \to B'' \to Z]$ in the diagram
\[
\begin{tikzcd}
& & A'' \ar[dl] \ar[r] & W \ar[dl]\ar[dr]
\ar[r] & B'' \ar[dr] & & \\
& A \ar[dl] \ar[r]
& B \ar[dr] & & A' \ar[dl] \ar[r]
& B' \ar[dr] & \\
X & & & Y & & & Z
\end{tikzcd}
\]
where $B'' = \{(b', s) \mid s: {q'}^{-1}(b') \to B, rs = p'\}$,
\[W = B'' \times_{B'} A' =
\{(a', s) \mid s: {q'}^{-1}(q'(a')) \to B, rs = p'\},\]
$W \to B$ is given by $(a', s) \mapsto s(a')$, and
$A'' = W \times_B A$
(see \cite{strickland2012tambarafunctors}).

There are three distinguished types of morphisms in
$\bispan_G$. For any map $f: X \to Y$ of finite $G$-sets,
we write
\[
R_f = [Y \xleftarrow{f} X \xrightarrow{1} X \xrightarrow{1} X],
\quad
N_f = [X \xleftarrow{1} X \xrightarrow{f} Y \xrightarrow{1} Y],
\quad
T_f = [X \xleftarrow{1} X \xrightarrow{1} X \xrightarrow{f} Y].
\]
These are called \textit{restriction,} \textit{norm,} and
\textit{transfer} along $f$, respectively.

Let $g \in G$.
for any $H \leq G$, there is
a $G$-isomorphism $f_g: G/H \to G/gHg^{-1}$ defined by
$xH \mapsto xHg^{-1} = xg^{-1}(gHg^{-1})$. Transfer along $f_g$
is denoted by $C_g$, and in fact
\[
C_g = T_{f_g} = R_{f_g}^{-1} = N_{f_g}.
\]

\subsection{Mackey and Tambara Functors}

\begin{defn}
A \textit{$G$-semi-Mackey functor} is a product-preserving functor
$\bispan_G^+ \to \cat{Set}$. A \textit{$G$-semi-Tambara functor}
is a product-preserving functor $\bispan_G \to \cat{Set}$.
\end{defn}

Products in $\bispan_G$ and $\bispan_G^+$ are defined by the disjoint
union on the level of objects.
It is well-known that every semi-Mackey functor takes values
in commutative monoids, and that every semi-Tambara functor takes
values in commutative semirings.

\begin{defn}
A \textit{$G$-Mackey functor} is a semi-Mackey functor which takes
values in (abelian) groups.
A \textit{$G$-Tambara functor} is a semi-Tambara
functor which takes values in (commutative) rings.
\end{defn}

In particular, the additive completion of a semi-Mackey functor is
a Mackey functor, and the additive completion of a semi-Tambara
functor is a Tambara functor
(see \cite[\S6]{Tambara01011993}).

There is an equivalent description of Mackey and Tambara
functors using only the transitive $G$-sets, which we now describe.

\begin{defn}
A \textit{$G$-Mackey functor} $M$ consists of the following data:
\begin{enumerate}
\item For each subgroup $H \leq G$, an abelian group
$M(G/H)$.

\item For each $g \in G$ and $H \leq G$, an isomorphism
$c_g: M(G/H) \to M(G/gHg^{-1})$ (the dependence on $H$ is
suppressed in the notation)
such that $c_h: M(G/H) \to M(G/H)$ is the
identity map for all $h \in H$ and
$c_g: M(G/H) \to M(G/gHg^{-1})$,
$c_{g'}: M(G/gHg^{-1}) \to M(G/g'gH(g'g)^{-1})$ satisfy
$c_{g'} c_g = c_{g'g}$ for all $g, g' \in G$.

\item For each subgroup inclusion $H \leq K$, group
homomorphisms $\Res_H^K: M(G/K) \to M(G/H)$ and
$\Tr_H^K: M(G/H) \to M(G/K)$ such that
\begin{enumerate}[label=(\roman*)]
\item $\Res_H^H = \Tr_H^H = \id$ for all $H$.
\item $\Res_H^K \Res_K^L = \Res_H^L$ and
$\Tr_K^L \Tr_H^K = \Tr_H^L$ whenever $H \leq K \leq L$.
\item For all $g \in G$,
\[
c_g \Res_H^K = \Res_{gHg^{-1}}^{gKg^{-1}} c_g,
\quad c_g \Tr_H^K = \Tr_{gHg^{-1}}^{gKg^{-1}} c_g
\]
whenever $H \leq K$.
\item (Double coset formula) Whenever $H, K \leq L$,
\[
\Res_K^L \Tr_H^L = \sum_{g \in K \backslash L/ H}
\Tr_{K \cap gHg^{-1}}^K c_g \Res_{H \cap g^{-1}Kg}^H.
\]
\end{enumerate}
\end{enumerate}
A semi-Mackey functor is obtained by relaxing the condition
that each $M(G/H)$ is an abelian group, requiring only
a commutative monoid instead.
\end{defn}

\begin{defn}
A \textit{$G$-Tambara functor $T$} consists of the following data:
\begin{enumerate}
\item For each subgroup $H \leq G$, a commutative ring
$T(G/H)$.

\item For each subgroup inclusion $H \leq K$, maps
$\Res_H^K: T(G/K) \to T(G/H)$, $\Tr_H^K: T(G/H) \to T(G/K)$,
and $\Nm_H^K: T(G/H) \to T(G/K)$, and for each
$g \in G$, isomorphisms $c_g: T(G/H) \to T(G/gHg^{-1})$ such that
\begin{enumerate}[label=(\roman*)]
\item The family of maps $(\{\Res_H^K\}, \{\Tr_H^K\}, \{c_g\})$
turns $T$ into a Mackey functor with respect to the additive
group structure.

\item The family of maps $(\{\Res_H^K\}, \{\Nm_H^K\}, \{c_g\})$
turns $T$ into a semi-Mackey functor under the multiplicative
monoid structure.
\end{enumerate}

\item (Exponential formula) For any exponential diagram
\[
\begin{tikzcd}
X \dar{} & A \lar{} & \lar{} X \times_Y \Pi_f A \dar{} \\
Y & & \Pi_f A \ar[ll]
\end{tikzcd}
\]
of finite $G$-sets, the following diagram commutes:
\[
\begin{tikzcd}
T(X) \ar[d, "\Nm", swap] & T(A) \ar[l, "\Tr",swap] \ar[r, "\Res"] &
T(X \times_Y \Pi_f A) \dar{\Nm} \\
T(Y) & & T(\Pi_f A) \ar[ll, "\Tr"]
\end{tikzcd}
\]
See \cite[\S1]{Tambara01011993} for more on exponential
diagrams.
\end{enumerate}
\end{defn}

\begin{rem}
The maps $\Res_H^K, \Tr_H^K, \Nm_H^K$ are the maps
induced by restriction, transfer,
and norm along the projection $G/H \to G/K$, while $c_g$
is the map induced by $C_g$. Note that the maps
$c_g$ induce an action of $W_H \coloneqq N(H)/H$ on $T(G/H)$, where
$N(H)$ is the normalizer of $H$ in $G$; $W_H$ is the
\textit{Weyl group} of $H$.
\end{rem}

\begin{rem}
Only (c) is not expressed precisely in terms of the data
we have given, as there is no
succinct way to express the exponential formula using
only transitive $G$-sets. When $X$ is not transitive, say
$X = \coprod_i G/H_i$, we extend the definition
of $T$ by setting $T(X) = \prod_i T(G/H_i)$;
there is also a way to extend the transfer and norm maps
using only the data given on the transitive sets.
A consequence of the exponential formula we will use
repeatedly is Frobenius reciprocity,
\[
x \Tr_H^K(y) = \Tr_H^K ( \Res_H^K(x) y)
\]
whenever defined.
\end{rem}

\begin{rem}
In most examples, we will
represent the data of a Mackey or Tambara functor
via its \textit{Lewis diagram,} a diagram of all the $T(X)$ for
transitive $X$ with only the maps $\Res_H^K, \Tr_H^K,
\Nm_H^K, c_g$
labeled.
\end{rem}

\begin{rem}
We will occasionally reference the notion of a
(commutative) \textit{Green functor,}
a Mackey functor taking values in commutative rings such that the
$\Res_H^K$ and $c_g$ are ring maps, and such that transfer
and restriction satisfy the Frobenius
reciprocity relations above. In particular, the
data of a Green functor
does not come with norm maps $\Nm_H^K$.
\end{rem}

\section{Free Polynomial Tambara Functors}

\subsection{Basic Properties of Free Tambara Functors}

For any two finite $G$-sets $X, Y$,
the set of morphisms $\bispan_G(X,Y)$ has a semiring structure, with
addition given by
\[
[X \ot A \to B \to Y] + [X \ot A' \to B' \to Y]
= [X \ot A \sqcup A' \to B \sqcup B' \to Y],
\]
multiplication given by
\[
[X \ot A \to B \to Y] \cdot [X \ot A' \to B' \to Y]
= [X \ot (A \times_Y B') \sqcup (A' \times_Y B)
\to B \times_Y B' \to Y],
\]
and additive and multiplicative identities given by
\[
[X \ot \varnothing \to \varnothing \to Y], \quad
[X \ot \varnothing \to Y \to Y]
\]
respectively (see \cite[Proposition 7.6]{Tambara01011993}).

\begin{defn}
Given a finite $G$-set $X$, the
\textit{free polynomial Tambara functor} or the
\textit{polynomial Tambara functor}
on $X$, denoted $\A[X]$, is the additive completion of
the representable functor
\[
\bispan_G(X, -): \bispan_G \to \cat{Set}.
\]
\end{defn}

In particular, we have a natural isomorphism
\[
\Hom_{G\cat{Tamb}}(\A[X], T) \cong
T(X)
\]
for any $G$-Tambara functor $T$ and finite $G$-set $X$.
When $X = \varnothing$, $\A \coloneqq \A[\varnothing]$
is the \textit{Burnside Tambara functor,} the initial
object in the category of Tambara functors.

We cite some basic facts about the structure maps
in a polynomial Tambara functor.

\begin{lem}[{\cite[Proposition 3.32]{Hill_2022}}]
\label{lem:transfer_and_restriction}
Let $H \leq G$, $f: Y \to Z$ be a map
of finite $G$-sets. The action
of post-composition by $T_f$ is defined
on elements of $\bispan_G(G/H, Y)$ by sending
\[
[G/H \ot A \to B \xrightarrow{r} Y] \mapsto
[G/H \ot A \to B \xrightarrow{f \circ r} Z],
\]
while post-composition by $R_f$ is defined
by sending
\[
[G/H \ot A \to B \to Z]
\mapsto
[G/H \ot A \times_Z Y \to B \times_Z Y \to Y].
\]
\end{lem}


Let $H$ be a subgroup of $G$. In order to understand
$\A[G/H]$, it suffices to understand the structure of
$\A[G/H](G/L)$ for $L \leq G$.
By splitting up the direct summands of a polynomial
$G/H \to G/L$, we see that every element of
$\A[G/H](G/L)$ is uniquely
a sum of (isomorphism classes of) polynomials of the form
\[
G/H \ot \coprod_i G/H_i \to G/K \to G/L,
\]
and their formal inverses,
i.e. polynomials $G/H \ot A \to B \to G/L$
where $B$ is transitive. We will call such a polynomial
\textit{irreducible.} It is clear that isomorphism classes
of irreducible polynomials form a $\Z$-basis for
$\A[G/H](G/L)$.

Recall the following definition:

\begin{defn}
A finite group $G$ is said to be a \textit{Dedekind group}
if all of its subgroups are normal.
\end{defn}

The only facts about Dedekind groups which will be
relevant to us are the following.
\begin{enumerate}
\item There exists a $G$-equivariant map
$G/H \to G/K$ iff $H \leq K$. In this case, the equivariant map
$G/H \to G/K$ sending $H$ to $gK$ is well-defined
for all $g \in G$.

\item For any two subgroups $H, K \leq G$, $HK = KH$
is also a subgroup.

\item All double cosets $H \backslash G /K$ are the same
as the one-sided cosets $G/HK = HK \backslash G$.
\end{enumerate}

When $G$ is a Dedekind group, any irreducible
polynomial of the form
\[
G/H \ot \coprod_i G/H_i \to G/K \to G/L
\]
must satisfy $H_i \leq K \leq L$ and $H_i \leq H$.
Furthermore, by composing appropriate elements of $G$
to $G/K$ and each factor of $G/H_i$, we can arrange
for the representative for the
isomorphism class of the diagram above to be
of the form
\[
\begin{tikzcd}
G/H & \coprod_i G/H_i \ar[l, "\sqcup_i f_i", swap]
\ar[r, "\sqcup_i \pr"] & G/K \ar[r, "\pr"]
& G/L
\end{tikzcd}
\]
where $\pr$ is generic notation for the natural projection
$G/S \to G/T$ whenever $S \leq T$ are subgroups
and $f: G/H \to G/K$ denotes the equivariant map
sending $H$ to $fK$.
Thus, an element $\A[G/H](G/L)$ can be represented by some tuple
$((H_1, f_1), \ldots, (H_n, f_n))_K$, where $f_i$ are elements of
$G/H$ which keep track of the image of $H_i$ under
the component $G/H_i \to G/H$.
Note here that the order of the
$(H_i, f_i)$ does not matter,
as shuffling the order gives an isomorphic
polynomial.

\begin{lem}\label{lem:equivalence}
Let $G$ be a finite Dedekind group.
Two tuples $((H_1, f_1), \ldots, (H_n, f_n))_K$ and
$((H_1', g_1), \ldots, (H_m', g_n))_K$
yield isomorphic polynomials iff $K = K'$, $m = n$, we have
$H_i' = H_i$ up to reordering, and the following condition
is true: there exists some $\ell \in L/K$ and lifts
$\ell_1, \ldots, \ell_n \in L$ such that
\[
\ell_i f_i = g_i
\]
up to reordering.
\end{lem}

\begin{proof}
The condition that $K = K'$ and $\{H_i\} = \{H_i'\}$ is obvious,
since $G/K \cong G/K'$ as $G$-sets iff $K = K'$. Thus we want
to determine when there
exists an isomorphism of irreducible polynomials
\[
\begin{tikzcd}
& \coprod_i G/H_i \ar[dl, "\sqcup_i f_i", swap] \ar[dd, "\sqcup r_i"]
\ar[r, "\sqcup_i \pr"] & G/K \ar[dr, "\pr"] \ar[dd, "s"] & \\
G/H & & & G/L \\
& \coprod_i G/H_i \ar[ul, "\sqcup_i g_i"] \ar[r, "\sqcup_i \pr", swap]
& G/K \ar[ur, "\pr", swap] &
\end{tikzcd}
\]
We see that in order for the diagram to commute, we must have
$s \in L/K$, $r_i \in L/H_i$ must lift $s \in L/K$, and
$f_i = r_ig_i$.
\end{proof}

\begin{rem}
This equivalence relation is hard to describe cleanly
and is a source of difficulty for tracking the combinatorics
involved in analyzing the ring structure of $\A[G/H]$.
We will later isolate some special cases for which
this task is easier.
\end{rem}

\subsection{A Levelwise Grading
for Polynomials}

The polynomial Tambara functors admit a levelwise $\N$-grading
for arbitrary finite $G$; when $G$ is a Dedekind group,
an explicit computation offers a slight refinement.

\begin{defn}
Let $G$ be a finite group
and $b = [G/H \ot A \to B \to G/L]$ be a nonzero irreducible polynomial,
i.e. one where $B$ is transitive
and nonempty. We define the \textit{degree}
of $b$ to be the value $\abs{A}/\abs{B}$.
\end{defn}

From this, we see that $\A[G/H](G/L)$ admits
an $\N$-grading $\A[G/H](G/L) \cong \bigoplus_{d \geq 0} S_d$,
where $S_d$ is the subgroup generated by the irreducible
polynomials of degree $d$. We now develop some basic results
that will allow us to describe the homogeneous elements concretely.

\begin{lem}\label{lem:degree_of_fiber}
Let $G$ be a finite group and let $A, B$ be finite
$G$-sets, with $B$ transitive and nonempty.
Suppose $f: A \to B$
is a map of $G$-sets. Then every fiber of $f$
has cardinality $\abs{A}/\abs{B}$.
\end{lem}

\begin{proof}
Suppose $A = G/H$ and $B = G/K$ with
$H \leq K$, and moreover that $f$ is the natural projection
map $G/H \to G/K$ defined by $gH \mapsto gK$.
The general case will then follow
upon writing $A$ as a disjoint union of
transitive $G$-sets and observing that
any map of $G$-sets $G/H \to G/K$ is can be decomposed
into a map as above followed by an isomorphism.

We claim
that the cardinality of every fiber of $f$
is $\abs{K}/\abs{H} = \abs{A}/\abs{B}$. Indeed,
for $g \in G$, the fiber over $gK \in G/K$
consists of those cosets $xH \in G/H$ such that
$g^{-1}x \in K$. Choosing a set of representatives
$x_1 H, \ldots, x_r H$ for $K/H$, these are precisely
the cosets of the form $gx_1 H, \ldots, gx_r H$.
\end{proof}

We also make the following definition:

\begin{defn}
Let $f: A \to B$ be a map of finite $G$-sets. If $B$
is nonempty and every
fiber of $f$ has the same cardinality $n$, we say that
\textit{$f$ has degree $n$} and write $\deg(f) = n$.
\end{defn}

\begin{lem}
A polynomial $b = [G/H \ot A \to B \to G/L]$
is homogeneous under the grading above
if the map $A \to B$ has a well-defined degree.
\end{lem}

\begin{proof}
Follows from Lemma \ref{lem:degree_of_fiber}.
\end{proof}

Thus, the degree of a polynomial
$b = [G/H \ot A \to B \to G/L]$ where $B$ is not
transitive is well-defined provided that
$A \to B$ has a well-defined degree.

Before moving on to Dedekind groups, we note a result
here that we will need for future computations.

\begin{lem}\label{lem:orbit_decomposition}
Let $G$ be a finite group, $L \leq G$, and $H, K \leq L$.
There exists an isomorphism
\[
G/H \times_{G/L} G/K \cong \coprod_{x \in H \backslash L / K}
G/(H \cap xKx^{-1})
\]
where the coproduct varies over choices of representatives
for the double cosets in $H \backslash L / K$.
\end{lem}

\begin{proof}
Observe that $G/H \times_{G/L} G/K$
is a union of orbits of the form $G \cdot (H, sK)$,
where $s \in L$. Two such orbits $G \cdot (H, sK)$
and $G \cdot (H, s'K)$ coincide if there exists some element
$h \in H$ such that $hsK = s'K$, which is equivalent to saying that
$s$ and $s'$ represent the same double coset in $H \backslash L / K$.
Since the stabilizer of $(H, sK)$ is $H \cap sKs^{-1}$, the result
follows.
\end{proof}

\begin{defn}
For any finite group $G$, let $\Oh_G^w$ consist of the pairs
\[
(n, K)
\]
where $K \leq G$ and $n \in \Z_{\geq 0}$ is a nonnegative
integer.
$\Oh_G^w$ assembles into a commutative monoid under the operation
\[
(n, K) + (m, L) = (n + m, K \cap L).
\]
In other words, $\Oh_G^w \cong \N \times \Oh_G$ as a monoid,
where $\Oh_G$ is the monoid of subgroups
of $G$ under intersection.
\end{defn}

\begin{defn}
When $G$ is Dedekind,
define the \textit{refined degree}
of an irreducible polynomial \[((H_1, f_1), \ldots,
(H_n, f_n))_K\]
to be
\[
\left( \sum_i \frac{\abs{K}}{\abs{H_i}}, K \right) \in
\Oh_L^w \subseteq \Oh_G^w.
\]
Here, the integer $\sum_i \abs{K}/\abs{H_i}$ is
again the degree of the map
\[
\coprod_i G/H_i \to G/K.
\]
In this case, the refined degree of a polynomial
$b = [G/H \to A \to B \to G/L]$ where
$B$ is not transitive is well-defined provided that
$A \to B$ has a well-defined degree,
and when writing $B = \coprod_i G/K_i$, we have
$K_i = K_j$ for all $i, j$.

When working with Dedekind groups, we will simply refer
to the refined degree as the \textit{degree} and refer
to the integer $\sum_i \abs{K}/\abs{H_i}$
as the \textit{numerical degree} when there is potential
for ambiguity.
\end{defn}

For the rest of this section, we will assume by default
that $G$ is Dedekind; however, most of the
results are analogous in the general case and can be obtained by
simply dropping the non-numerical component of $\Oh_L^w$.

It follows immediately from definition that
\[
\A[G/H](G/L) \cong \bigoplus_{d \in \Oh_L^w} S_d,
\]
as groups,
where $S_d$ is the group generated by the irreducible
polynomials of degree $d \in \Oh_L^w$. This also gives
a description of $\A[G/H]/(G/L)$ as a graded ring.

\begin{lem}
$((H_1, f_1), \ldots, (H_n, f_n))_K \cdot
((H_1', g_1), \ldots, (H_m', g_m))_{K'}$ is
a homogeneous element of degree
\begin{align*}
&\deg ((H_1, f_1), \ldots, (H_n, f_n))_K
+ \deg ((H_1', g_1), \ldots, (H_m', g_m))_{K'} \\
&= \left( \sum_i \frac{\abs{K}}{\abs{H_i}}
+ \sum_j \frac{\abs{K'}}{\abs{H_j'}},
K \cap K' \right)
\end{align*}
\end{lem}

\begin{proof}
Recall that multiplication on polynomials is given by
\[
[X \ot A \to B \to Y] \cdot [X \ot A' \to B' \to Y]
= [X \ot (A \times_Y B') \sqcup (A' \times_Y B)
\to B \times_Y B' \to Y].
\]
If the cardinality of every fiber of
$A \to B$ is $n$, then so is that of every fiber of
$A \times_Y B' \to B \times_Y B'$, as this value
is preserved by pullbacks;
the same holds for $A' \to B'$. Hence every fiber of the
map
$(A \times_Y B') \sqcup (A' \times_Y B)
\to B \times_Y B'$
has cardinality $n + m$, and the fact that the associated
subgroup is $K \cap K'$ follows from
Lemma \ref{lem:orbit_decomposition}.
\end{proof}

\begin{rem}
Note that the elements of numerical degree $0$ form a subring
which we temporarily
denote by $\A_0[G/H](G/L)$. This ring is generated
by irreducible polynomials of the form
\[
G/H \ot \varnothing \to G/K \to G/L,
\]
so we can identify the $\Z$-basis with the poset $\Oh_L$ of
subgroups $K \leq L$.

In particular, $\A_0[G/H](G/L)$
has finite $\Z$-rank.
Note moreover that since polynomials of the form
\[
G/H \ot \varnothing \to X \to G/L
\]
where $X \to G/L$ is an equivariant map of finite $G$-sets
are in bijection with polynomials $\varnothing \to G/L$,
we have an isomorphism between $\A_0[G/H](G/L)$ and
$\A(G/L)$, where $\A$ is the Burnside ring
$\A = \A[\varnothing]$.
In other words, the map
$\bispan_G(\varnothing, -) \hookrightarrow
\bispan_G(G/H, -)$ obtained via precomposition
by the unique morphism $G/H \to \varnothing$ in
$\bispan_G$ induces a morphism
$\A[\varnothing] \to \A[G/H]$ that is injective and
has as its levelwise
image the elements of numerical degree $0$.
\end{rem}



\begin{rem}
The homogeneous
elements of degree $(-, L)$ also form a subring of
$\A[G/H](G/L)$, which we denote by $\A^0[G/H](G/L)$. This ring is
generated by irreducible polynomials of the form
\[
G/H \ot \coprod_i G/H_i \to G/L \to G/L.
\]
In particular,
the homogeneous
elements of degree $(-, L)$ form an $\N$-graded subring, since
\[
(n, L) + (m, L) = (n + m, L)
\]
in $\Oh_L^w$. It is easy to describe the multiplication of
such irreducible polynomials: we have
\[
((H_i, f_i)_{i \in I})_L \cdot ((H_j', g_j)_{j \in J})_L
= ((H_i, f_i)_{i \in I}, (H_j', g_j)_{j \in J})_L.
\]
In particular, we see that $\A^0(G/L)$ is a
finite-type $\Z$-algebra.
The equivalence relation in Lemma \ref{lem:equivalence} is also
simple; the set of representatives
for the equivalence is given by orbits of collections
$((H_1, f_1), \ldots, (H_n, f_n))_L$ under the $L^n$-action
\[
((H_1, f_1), \ldots, (H_n, f_n))_L \mapsto
((H_1, \ell_1 f_1), \ldots, (H_n, \ell_n f_n)),
\]
so each of the $f_i$ is a well-defined element of
$G/HL$. While $\A^0$ is preserved by restriction and norm maps,
it is not preserved by transfer, so it is not a sub-Tambara
functor.
\end{rem}

\begin{ex}
For $G = C_p$, we have
$\A_0[C_p/C_p](C_p/C_p) \cong \Z[t]/(t^2 - pt)$,
where here $t$ is the element corresponding to the
polynomial
\[
C_p/C_p \ot \varnothing \to C_p/e \to C_p/C_p
\]
while $\A_0[C_p/C_p](C_p/e) \cong \Z$.
On the other hand, we have
$\A^0[C_p/C_p](C_p/e) = \Z[x]$, where here $x$ is the element
corresponding to the polynomial
\[
C_p/C_p \ot C_p/e \to C_p/e \to C_p/e.
\]
$\A^0[C_p/C_p](C_p/C_p) = \Z[x]$, where $x$ is the element
corresponding to the polynomial
\[
C_p/C_p \ot C_p/C_p \to C_p/C_p \to C_p/C_p
\]
\end{ex}

\begin{lem}
The Weyl group action on $\A[G/H](G/L)$ preserves degrees.
\end{lem}

\begin{proof}
Obvious.
\end{proof}

\begin{lem}\label{lem:tr_res_grading}
Let $L \leq L'$.
$\Tr: \A[G/H](G/L) \to \A[G/H](G/L')$ and
$\Res: \A[G/H](G/L') \to \A[G/H](G/L)$
are graded homomorphisms,
in the sense that there are monoid
homomorphisms $\phi: \Oh_L^w \to \Oh_{L'}^w$ and
$\psi: \Oh_{L'}^w \to \Oh_L^w$ such that
\[
\Tr\left( \A[G/H](G/L)_d \right) \subseteq
\A[G/H](G/L)_{\phi(d)}, \quad
\Res\left( \A[G/H](G/L')_{d'} \right) \subseteq
\A[G/H](G/L)_{\psi(d')}
\]
for any $d \in \Oh_L^w$, $d' \in \Oh_{L'}^w$. Explicitly,
\[
\phi(n, K) = (n, K),
\quad
\psi(n, K) = \left(n
, K \cap L\right).
\]
\end{lem}

\begin{proof}
Observe that $\phi, \psi$ indeed define monoid
homomorphisms.
The fact that $\Res$ and $\Tr$ are graded in this way
more or less follows immediately from the explicit description
of transfer and restriction in Lemma
\ref{lem:transfer_and_restriction},
the fact that pullback preserves degrees,
and the fact that $G/K \times_{G/L'} G/L
= \coprod_{i=1}^{[L'/KL]} G/(K \cap L)$.
\end{proof}



We summarize these observations in the following,
more general lemma.

\begin{lem}
Let $\mathcal{F}$ be a collection of subgroups of $L$ and let
$\Oh_{L, \mathcal{F}}^w$ be the subset of $\Oh_L^w$ consisting of
all pairs of the form
$(n, K)$
where $K \in \mathcal{F}$.
\begin{enumerate}
\item If $L \in \mathcal{F}$ and $\mathcal{F}$ is closed under
intersections, then the homogeneous elements of degree in
$\Oh_{L, \mathcal{F}}^w$ form a subring of $\A[G/H](G/L)$.

\item If for every $K \in \mathcal{F}$ and $K' \leq L$,
$K \cap K' \in \mathcal{F}$, then the homogeneous elements of degree in
$\Oh_{L, \mathcal{F}}^w$ form an ideal of $\A[G/H](G/L)$. In particular,
any family of subgroups $\mathcal{F}$ which is ``downwards closed''
in the sense that $K' \in \mathcal{F}$ whenever $K' \leq K$
and $K \in \mathcal{F}$ defines an ideal of $\A[G/H](G/L)$.

\item For $L \leq L'$, the image of transfer
$\Tr: \A[G/H](G/L) \to \A[G/H](G/L')$ is the ideal corresponding to the
family $\mathcal{F} = \{K \leq L\}$.
\end{enumerate}
\end{lem}

\begin{proof}
Obvious.
\end{proof}

\begin{rem}
The subring $\A^0[G/H](G/L)$ can be identified with the quotient
of $\A[G/H](G/L)$ by the ideal corresponding to the family
$\mathcal{F} = \{K \lneq L\}$.
\end{rem}

\begin{rem}
$\Nm_L^{L'}: \A[G/H](G/L) \to \A[G/H](G/L')$
is not generally graded.
\end{rem}

\begin{rem}\label{rem:naive}
None of the results on the grading of
$\A[G/H](G/L)$ depend on the transitivity
of $X = G/H$. In particular, for any finite $G$-set $X$, the
ring $\A[X](G/L)$ is graded over $\Oh_L^w$ in the way described:
an irreducible polynomial
\[
X \ot \coprod_i G/H_i \to G/K \to G/L
\]
has degree $( \sum_i \abs{K}/\abs{H_i}, K)$, which turns
$\A[X](G/L)$ into a graded ring for each $L \leq G$.
Restriction and transfer maps for $\A[X]$ are also graded
in the same way. For now, we call this the \textit{naive grading}
on $\A[X]$.
In \S\ref{sec:boxproduct}, we will show
that the gradings on $\A[G/H]$ for transitive $G/H$ induce a grading
on $\A[X]$ for arbitrary $G$-sets $X$, which coincides with the
naive grading.
\end{rem}

\section{Finite Generation for Dedekind Groups}
\label{sec:finite_generation}

\subsection{Finiteness Results for
General \texorpdfstring{$G$}{G}}

The main goal of this section is to prove that the
polynomial Tambara functors $\A[G/H]$ are levelwise finitely
generated for Dedekind $G$ and relatively
finite-dimensional for general $G$.
First, some definitions.

\begin{defn}
Let $T$ be a $G$-Tambara functor and let $R$ be a $T$-algebra,
i.e. $R$ is a $G$-Tambara functor equipped with a morphism
$f: T \to R$. We say that $R$ is \textit{levelwise finitely
generated over $T$} if each $R(G/H)$ is a finite-type
$T(G/H)$-algebra. When $T = \A = \A[\varnothing]$ is the
Burnside Tambara functor, we simply say that $R$ is
\textit{levelwise finitely generated.}
\end{defn}

\begin{rem}
Since $\A$ is levelwise a finite $\Z$-module,
$R$ is levelwise finitely generated iff each $R(G/H)$ is a
finite-type $\Z$-algebra.
\end{rem}

\begin{defn}
[{\cite[Definition 3.30]{chanwisdom2025algebraicktheory}}]
A Green or Tambara functor $R$ is \textit{relatively
finite-dimensional} if for all $H \leq K \leq G$, the
restriction map $\Res_H^K: R(G/K) \to R(G/H)$
is a finite ring map. Equivalently,
$\Res_H^G: R(G/G) \to R(G/H)$ is a finite ring map for all
$H \leq G$.
\end{defn}

Relative finite-dimensionality is a well-behaved
finiteness condition on Green and Tambara functors
which imposes many niceties on their category of modules
(see \cite[\S{3}]{chanwisdom2025algebraicktheory}
for more details);
we will use it extensively in \S\ref{sec:boxproduct}.
Next, we make some remarks on finite
generation for general finite groups $G$.

\begin{thm}[{\cite[Theorem A]{BRUN2005233}}]
\label{thm:brun}
Let $G$ be a finite group, $X$ a finite $G$-set. Then
$\A[X](G/e)$ is a free polynomial ring over $\Z$
on $\abs{X}$ generators.
\end{thm}

In particular, $\A[X](G/e)$ is a finite-type $\Z$-algebra.
We also cite a fact about Tambara functors we will use
repeatedly.

\begin{lem}[{\cite[Proposition 5.1]{Tambara01011993}}; also see
{\cite[Lemma 3.3]{schuchardt2025algebraicallyclosedfieldsequivariant}}]
\label{lem:integral}
Let $G$ be a finite group and $T$
a $G$-Tambara functor. For all $H \leq K$, restriction
$\Res_H^K: T(G/K) \to T(G/H)$ is an integral ring map.
\end{lem}

\begin{defn}
Let $G$ be a finite group, $H, L \leq G$. Let $\A^0[G/H](G/L)$
denote the subring of $\A[G/H](G/L)$ generated by those irreducible
polynomials
\[
G/H \ot A \to B \to G/L
\]
such that $B \to G/L$ is an isomorphism.
\end{defn}

This is indeed a subring, since if $B \to G/L$ and $B' \to G/L$
are isomorphisms, then so is $B \times_{G/L} B' \to G/L$.
Note that for Dedekind $G$, this coincides with our previous
definition of $\A^0$.

\begin{lem}\label{lem:A0_finite}
$\A^0[G/H](G/L)$ is a finitely generated $\Z$-algebra.
\end{lem}

\begin{proof}
An irreducible polynomial in $\A^0[G/H](G/L)$ is of the form
\[
G/H \ot \coprod_i G/H_i \to G/g^{-1}Lg \to G/L.
\]
By replacing the $G/H_i$ with isomorphic $G$-sets, we may
assume that all $H_i \leq g^{-1}Lg$, so that it has a representative
of the form
\[
G/H \xleftarrow{\sqcup_i f_i} \coprod_i G/H_i
\to G/g^{-1}Lg \xrightarrow{g^{-1}} G/L,
\]
where $f_i: G/H_i \to G/H$ is the unique $G$-map sending
$H_i \in G/H_i$ to $f_i H \in G/H$ (here $H_i \leq
f_i H f_i^{-1}$)
and each $G/H_i \to G/g^{-1}Lg$
is the natural projection.
Multiplication of two such polynomials gives
\begin{align*}
&[G/H \xleftarrow{\sqcup_i f_i} \coprod_i G/H_i
\to G/g^{-1}Lg \xrightarrow{g^{-1}} G/L]\\
&\quad
\cdot
[G/H \xleftarrow{\sqcup_j g_j} \coprod_j G/H_j'
\to G/g^{-1}Lg \xrightarrow{g^{-1}} G/L] \\
&=[G/H \xleftarrow{\sqcup_i f_i \sqcup_j g_j}
\coprod_{i, j} G/H_i \sqcup G/H_j'
\to G/g^{-1}Lg \xrightarrow{g^{-1}} G/L],
\end{align*}
where we are using the fact that
$G/g^{-1}Lg \times_{G/L} G/g^{-1}Lg = G/g^{-1}Lg$
and
$G/H_i \times_{G/L} G/g^{-1}Lg
= G/H_i \times_{G/g^{-1}Lg} G/g^{-1}Lg = G/H_i$. Hence,
we can take as a set of generators all irreducible
polynomials of the form
\[
G/H \ot G/H' \to G/g^{-1}Lg \to G/L
\]
for varying $g \in G$ and $H' \leq g^{-1}Lg$.
\end{proof}

\begin{lem}\label{lem:relatively_finite}
$\A[G/H]$ is relatively finite-dimensional.
\end{lem}

\begin{proof}
Write $T = \A[G/H]$ and $T^0(G/L) = \A^0[G/H](G/L)$ for brevity.
We prove that $\Res_L^G$ is a finite ring map by induction
on the size of $L$. Since $T(G/e)$ is a finite-type
$\Z$-algebra by Theorem \ref{thm:brun}, $\Res_e^G$ is a finite-type
ring map, hence a finite ring map since
all restrictions maps are integral (Lemma \ref{lem:integral}).

Now suppose $L \leq G$. It is clear from definition that
every irreducible polynomial in $T(G/L)$ is either in
$T^0(G/L)$ or is of the form $\Tr_K^L(x)$
for some $x \in T(G/K)$, $K \lneq L$.
By Frobenius reciprocity,
$\Res_L^G(y)$ where $y \in T(G/G)$ acts by multiplication on
$\Tr_K^L(x)$ via
\[
\Res_L^G(y) \Tr_K^L(x) = \Tr_K^L(\Res_K^G(y) x),
\]
so it follows that as a $T(G/G)$-module,
$T(G/L)$ is generated by $T^0(G/L)$
and modules isomorphic to quotients of
$T(G/K)$ for $K \lneq L$.
By assumption, the $T(G/K)$ are all
finite $T(G/G)$-modules, so it suffices to show that
the $T(G/G)$-module generated by $T^0(G/L)$ in $T(G/L)$
is contained in some finite $T(G/G)$-module as well.
But $T^0(G/L)$ is a finite-type $\Z$-algebra,
so it generates a finite-type subalgebra over
$T(G/G)$ in $T(G/L)$. Since this subalgebra is integral
over $T(G/G)$ by Lemma \ref{lem:integral} again, it must
be finite over $T(G/G)$, completing the proof.
\end{proof}

\begin{cor}\label{cor:general_reduction}
$\A[G/H]$ is levelwise finitely generated over $\Z$
iff $\A[G/H](G/G)$ is a finite-type $\Z$-algebra.
\end{cor}

\begin{proof}
The only if direction is obvious. The converse follows
from the fact that any algebra which is module-finite over
a finite-type $\Z$-algebra is also a finite-type $\Z$-algebra.
\end{proof}

\subsection{Dedekind
\texorpdfstring{$G$}{G},
Transitive Case}

We introduce some auxiliary terms on
levelwise finite generation.

\begin{defn}\label{def:tambara_finite}
Let $G$ be a finite group. We say that $G$ is
\textit{Tambara finite} if $\A[X]$ is levelwise finitely
generated for all finite $G$-sets $X$, and \textit{transitively
Tambara finite} or just \textit{transitively finite} if
$\A[X]$ is levelwise finitely generated for all transitive
finite $G$-sets $X$.
\end{defn}

We will show in \S\ref{sec:boxproduct} that transitively
Tambara finite
groups are Tambara finite.
For the rest
of this section, assume that $G$ is a finite Dedekind group;
our goal now is to show (Theorem \ref{thm:finite_generation})
that $G$ is transitively finite.

\begin{lem}
Let $H \leq H'$. The morphism of Tambara functors
$\A[G/H] \to \A[G/H']$ induced by restriction along the
projection morphism $G/H \to G/H'$
contains in its levelwise image the set of irreducible polynomials
$((H_i, f_i)_{i \in I})_K$ where $H_i \leq H$.
\end{lem}

\begin{proof}
This follows because the composition of the polynomials
\[
G/H' \xleftarrow{\pr}
G/H \to G/H \to G/H, \quad G/H \xleftarrow{\sqcup f_i}
\coprod_i G/H_i \to
G/K \to G/L
\]
is the polynomial
\[
G/H' \xleftarrow{\pr \circ \sqcup f_i} \coprod_i G/H_i \to G/K
\to G/L. \qedhere
\]
\end{proof}

The punchline is that for $G$ Dedekind,
the $H_i$ must all collectively
factor through $H \cap K$---thus, for the finite generation
of $\A[G/H](G/L)$,
we only need to consider (by induction on the size of $H$)
those irreducible polynomials with $H \leq K \leq L$
(so in particular, $H \leq L$).

\begin{lem}
When $H \leq K \leq L$,
the isomorphism class of the polynomial
$((H_i, f_i)_{i \in I})_K$ in $\A[G/H](G/L)$
is determined by the reduction of the
$f_i$ mod $K/H$.
\end{lem}

\begin{proof}
Apply Lemma \ref{lem:equivalence} with $s = 1$.
\end{proof}

Thus, we can view our polynomials $((H_i, f_i)_{i \in I})_K$ as having
$f_i \in G/K$.

\begin{lem}\label{lem:multiplication_special}
Let $H \leq K \leq K' \leq L$. Then
\[
((H_i, f_i)_{i \in I})_K \cdot ((H_j', g_j)_{j \in J})_{K'}
= \sum_{s \in L/K'} \left((H_i, f_i)_{i \in I} \sqcup
\coprod_{\substack{j \in J \\ r \in L/K \\ r \mapsto s}}
(H_j', r g_j) \right)_K
\]
in $\A[G/H](G/L)$,
where here $((H_1, f_1) \sqcup (H_2, f_2))$ just means the pair
$((H_1, f_1), (H_2, f_2))$.
\end{lem}

\begin{proof}
We will unravel the multiplication rule
\[
[X \ot A \to B \to Y] \cdot [X \ot A' \to B' \to Y]
= [X \ot (A \times_Y B') \sqcup (A' \times_Y B)
\to B \times_Y B' \to Y]
\]
in this case. Recall the orbit decomposition
\[B \times_Y B' = G/K \times_{G/L} G/K'
\cong \coprod_{s \in L/K'}
G/K\]
in Lemma \ref{lem:orbit_decomposition},
where the identification is given by noting that
$G/K \times_{G/L} G/K'$ is the union of orbits of the form
\[
G/K \times_{G/L} G/K'
= \coprod_{s \in L/KK'}
G \cdot (K, sK')
= \coprod_{s \in L/K'}
\frac{G}{K} \cdot (K, sK').
\]
We also have
\begin{align*}
A \times_Y B'
= \left( \coprod_i G/H_i \right) \times_{G/L}
G/K'
&= \coprod_i G/H_i \times_{G/L} G/K' \\
&= \coprod_i \coprod_{r \in L/K'}
\frac{G}{H_i} \cdot (H_i, rK'),
\end{align*}
\begin{align*}
A' \times_Y B
= \left( \coprod_j G/H_j' \right) \times_{G/L}
G/K
&= \coprod_j G/H_j' \times_{G/L} \times G/K \\
&= \coprod_j \coprod_{r \in L/K}
\frac{G}{H_j'} \cdot (rH_j', K).
\end{align*}
Which components of $A \times_Y B'$ and
$A' \times_Y B$ lie over the component $G/K$
corresponding to a given $s \in L/K'$?
These are the components of $A \times_Y B'$ with
\[
(H_i, rK') \mapsto (K, sK')
\]
and the components of $A' \times_Y B$ with
\[
(rH_j', K) \mapsto (K, sK').
\]
In the first case we need $r \in L/K'$
to map to $s \in L/K'$ and in the second case
we need $r \in L/K$ to map to $s \in L/K'$.

Under this identification, we then see that the product is a sum
of irreducible polynomials,
indexed over $s \in L/K'$, of the form
\[
G/H \ot
\left(\coprod_{i \in I}
G/H_i \right) \sqcup
\left( \coprod_{\substack{j \in J \\
r \in L/K \\ r \mapsto s}}
G/H_j' \right)
\to G/K \to G/L.
\]
Now we analyze each component $G/H_i \to G/H$
and $G/H_j' \to G/H$. For a given $i \in I$ and $s \in L/K'$,
the element in the component $G/H_i$ of $A \times_Y B'$
lying above the $s$ piece $G/K$
corresponding to $H_i \in G/H_i$ is
$(H_i, rK')$. The map to $G/H$ is then projection onto the first
coordinate followed by $f_i$, so we see that
$G/H_i \ni H_i \mapsto f_iH \in G/H$. Similarly,
for a given $j \in J$ and $r \in L/K$ mapping to $s \in L/K'$,
the element in $A' \times_Y B$ corresponding to $H_j' \in G/H_j'$
is $(rH_j', K)$. The map to $G/H$ is projection onto the first
coordinate followed by $g_j$, so
$G/H_j' \ni H_j' \mapsto rg_j H \in G/H$.
This completes the proof.
\end{proof}

\begin{thm}
\label{thm:finite_generation}
Let $G$ be a finite Dedekind group. Then $G$ is
transitively finite, i.e.
$\A[G/H]$ is levelwise finitely generated
for any $H \leq G$.
\end{thm}

\begin{proof}
By Corollary \ref{cor:general_reduction},
it suffices to show that $\A[G/H](G/G)$ is finitely generated.
We claim that $\A[G/H](G/G)$
is generated by the following elements:
\begin{enumerate}
\item Irreducible polynomials that lie in the levelwise image of any
map $\A[G/H'] \to \A[G/H]$ induced by restriction along the
projection morphism $G/H' \to G/H$, where $H' \lneq H$, and

\item Irreducible polynomials of the form
$((H_i, f_i)_{i \in I})_K$ that satisfy
the following property: $K \geq H$, and
for any $H' \leq K$,
the number of tuples $(H_i, f_i)$ appearing in
$((H_i, f_i)_{i \in I})_K$ with $H_i = H'$
is at most $\binom{\abs{G/K}}{2} + 1$.
\end{enumerate}



Fix a subgroup $K \geq H$ and a subgroup $H_0 \leq K$.
Let's introduce some notation. Given any irreducible
polynomial $p = ((H_j, f_j)_{j \in J})_K$ and
$f \in G/K$, let $m_f(p)$ denote the number of times that
$(H_0, f)$ appears in $p$.
Let $\vec{v}(p) \in \N^{G/K}$
(here $\N^{G/K}$ is the set of $\N$-valued vectors
indexed in elements of $G/K$)
denote the vector whose $f$-coordinate has value $m_f(p)$.
Similarly, given any vector $\vec{v} \in \N^{G/K}$, let
$m_f(\vec{v})$ denote the value of $\vec{v}$ at the $f$ coordinate.
Furthermore, there is an $G/K$ action on $\N^{G/K}$ given by
$m_f(g \cdot \vec{v}) = m_{g^{-1}f}(\vec{v})$.

If $\vec{w} \in \N^{G/K}$ with $\vec{w} \leq \vec{v}(p)$
($m_f(\vec{w}) \leq m_f(p)$ for all $f \in G/K$), write
$p/\vec{w}$ for the polynomial
obtained by deleting $m_f(\vec{w})$ copies
of $(H_0, f)$ from $p$ for each $f \in G/K$. If
$\vec{w} \in \N^{G/K}$ and $K' \geq K$,
we let $p_{K'}(\vec{w})$ denote the
irreducible polynomial
\[
p_{K'}(\vec{w}) = \left( \coprod_{f \in G/K}
\coprod_{m_f(\vec{w})} (H_0, f) \right)_{K'}
\]
i.e. $(H_0, f)$ appears in $p_{K'}(\vec{w})$ exactly $m_f(\vec{w})$
times.

Lastly, for any $\vec{v} \in \N^{G/K}$
and $p = ((H_j, f_j)_{j \in J})_K$,
write $p^{\text{edit}}(\vec{v})$ to be the irreducible
polynomial consisting of all the pairs $(H_j, f_j)$
with $j \in J$ and $H_j \ne H_0$, and that contains
exactly $m_f(\vec{v})$ copies of $(H_0, f)$ for each
$f \in G/K$. In other words,
$p^{\text{edit}}(\vec{v})$ is obtained from $p$ by
editing the multiplicities of each $(H_0, f)$ to be exactly
that described by $\vec{v}$.
In this notation, $p/\vec{w}
= p^{\text{edit}}(\vec{v}(p) - \vec{w})$ and
Lemma \ref{lem:multiplication_special}
implies
\[
p^{\text{edit}}(\vec{v}) \cdot p_{K'}(\vec{w})
= \sum_{s \in G/K'} p^{\text{edit}}\left( \vec{v} +
\sum_{\substack{r \in G/K \\ r \mapsto s}} r \cdot \vec{w} \right).
\]
for any $K' \geq K$.

What we will show is this: every irreducible polynomial
of the form $p = ((H_j, f_j)_{j \in J})_K$
with $\sum_{f \in G/K} m_f(p) > \binom{\abs{G/K}}{2} + 1$
can be expressed as a combination of terms of lower numerical degree,
generators of type (b), and irreducible polynomials of the form
$p^{\text{edit}}(\vec{v})$ where $\sum_{f \in G/K} m_f(\vec{v})
\leq \binom{\abs{G/K}}{2} + 1$. This is sufficient by induction
on the number of distinct subgroups $(H, -)$ that appear in a
polynomial $p$ more than $\binom{\abs{G/K}}{2} + 1$ times
and by the fact that generators of type (a)
and type (b) contain all irreducible polynomials of degree
$(n, K')$ for $n$ sufficiently low.

Our proof strategy is a generalization of the corresponding
proof for $G = C_p$ in \cite{chan2025tambaraaffineline}.
We adopt analogous notation:
for any $j \geq 0$ let
\[
S_j(\vec{v}) = \{f \in G/K: m_f(\vec{v}) = j\}
\]
and define $S_j(p) = S_j(\vec{v}(p))$
for any polynomial $p = ((H_j, f_j)_{j \in J})_K$. Similarly define
\[
T_j(\vec{v}) = \bigcup_{i \geq j} S_i(\vec{v}),
\quad T_j(p) = T_j(\vec{v}(p)).
\]
Given any subset $A \subseteq G/K$, let $\vec{e}(A)
\in \N^{G/K}$ be the vector with
\[
m_f(\vec{e}(A)) = \begin{cases} 1 & f \in A, \\ 0 & \text{otherwise.}
\end{cases}
\]
Note that if $A \subseteq G/K$ is a subset whose elements
have pairwise distinct images in $G/K'$, where $K' \geq K$, then
the polynomial $p_{K'}(\vec{e}(A))$ is a type (b) generator,
since $\abs{A} \leq \abs{G/K'} \leq \binom{\abs{G/K'}}{2} + 1$.
The other reason for choosing the type (b) generators as they are
is the following lemma.

\begin{lem}\label{lem:subclaim}
Let $p = ((H_j, f_j)_{j \in J})_K$ be an irreducible
polynomial such that the number of tuples
$(H_j, f_j)$ with $H_j = H_0$ in $p$ exceeds
$\binom{\abs{G/K}}{2} + 1$
(i.e. $\sum_{f \in G/K} m_f(p) > \binom{\abs{G/K}}{2} + 1$).
Then there exists an integer $j \geq 0$ such that
$S_j(p) = \varnothing$ but $S_{j+1}(p) \ne \varnothing$.
In particular, if there exists some $k \geq 0$ such
that $T_{k+1}(p) = 0$ but
$S_i(p) = \varnothing$ for all $i \leq k$,
then $\sum_{f \in G/K} m_f(p) \leq \binom{\abs{G/K}}{2} + 1$.
\end{lem}

\begin{proof}
Otherwise, there must exist an $r \geq 0$ such that
$S_j(p) \ne \varnothing$ for all $j \leq r$. Since
\[
\sum_{f \in G/K} m_f(p) = \sum_{j \geq 0} j \abs{S_j(p)}
\]
and $\sum_{j \geq 0} \abs{S_j(p)} = \abs{G/K}$,
this sum is maximized when $\abs{S_j(p)} = 1$
for all $j \leq \abs{G/K} - 1$, in which case
$\sum_{f \in G/K} m_f(p) \leq 1 + \cdots + (\abs{G/K} - 1)
= \binom{\abs{G/K}}{2}$,
contradicting our assumption for $p$.
The last statement follows from the fact that
if a $k$ as described exists, then there does not exist an integer
$j \geq 0$  with $S_j(p) = \varnothing$ and $S_{j+1}(p) \ne
\varnothing$.
\end{proof}

We use a nested induction argument, inducting
successively on $\abs{S_0(b)}, \abs{S_1(b)}, \ldots$, and so
forth, to prove the following claim:
if $p = ((H_j, f_j)_{j \in J})_K$ with
$\sum_{f \in G/K} m_f(p) > \binom{\abs{G/K}}{2} + 1$,
it may be expressed as a combination of terms
of lower numerical degree, generators of type (b), and irreducible
polynomials of the form $p^{\text{edit}}(\vec{v})$
where $\sum_{f \in G/K} m_f(\vec{v}) \leq \binom{\abs{G/K}}{2} + 1$.
We start with the base case
$\abs{S_0(b)} = 0$. If $T_0(b) = 0$, then we are trivially done.
Otherwise, we have
\[
p = (p/\vec{e}(G/K)) \cdot p_G(\vec{e}(\{1\}))
\]
by Lemma \ref{lem:multiplication_special}.
Note that $p/\vec{e}(G/K)$ is of lower numerical degree
(because it contains less tuples than $p$) and
$p_G(\vec{e}(\{1\}))$ is a type (b) generator
(it is of the form $p_{K'}(\vec{e}(A))$ described above).

We now establish the following inductive hypothesis:
\begin{enumerate}
\item[(0)] Suppose there exists an integer $m_0 > 0$ such that
any irreducible polynomial of the form
$p = ((H_j, f_j)_{j \in J})_K$ with $\abs{S_0(p)} < m_0$
can be expressed as a combination of generators, lower-degree
terms, and polynomials of the form $p^{\text{edit}}(\vec{v})$
where $\sum_{f \in G/K} m_f(\vec{v}) \leq \binom{\abs{G/K}}{2} + 1$.
\end{enumerate}
We prove the inductive step for $\abs{S_0(p)} = m_0$
by induction on $\abs{S_1(p)}$. At the base case, suppose
$\abs{S_1(p)} = 0$. If $T_1(p) = \varnothing$, then
by Lemma \ref{lem:subclaim},
$\sum_{f \in G/K} m_f(\vec{v}) \leq \binom{\abs{G/K}}{2} + 1$
and we are done. Otherwise,
we have $G/K = T_1(p) \sqcup S_0(p)$, where $T_1(p)$
is nonempty.

Let $K'/K$ be the stabilizer of $T_1(p)$ under
the action of $G/K$, so that $T_1$ is the union of cosets
\[
T_1(p) = (g_1K'/K) \sqcup \cdots \sqcup (g_t K'/K).
\]
By assumption, $K'/K \ne G/K$, since $T_1(b)$ is nonempty
and cannot be all of $G/K$, since $\abs{S_0(p)} = m_0 > 0$.
Lemma \ref{lem:multiplication_special}
then implies
\begin{align*}
p &= (p/\vec{e}(T_1(p))) \cdot
p_{K'}(\vec{e}(\{g_1, \ldots, g_t\})) \\
&\quad -
\sum_{1 \ne s \in G/K'} p^{\text{edit}}\left(\vec{v}(p) - \vec{e}(T_1(p))
+ \sum_{\substack{r \in G/K \\ r \mapsto s}}
r \cdot \vec{e}(\{g_1, \ldots, g_t\}) \right) \\
&= (p/\vec{e}(T_1(p))) \cdot
p_{K'}(\vec{e}(\{g_1, \ldots, g_t\})) \\
&\quad -
\sum_{1 \ne s \in G/K'} p^{\text{edit}}\left(\vec{v}(p) - \vec{e}(T_1(p))
+ s \cdot \vec{e}(T_1(p)) \right)
\end{align*}
Note that $p/\vec{e}(T_1(p))$ is a term of lower degree, while
$p_{K'}(\vec{e}(\{g_1, \ldots, g_t\}))$ is a type (b) generator
(again because
it is of the form $p_{K'}(\vec{e}(A))$ described above).
The first equality follows because
\[
\sum_{\substack{r \in G/K \\ r \mapsto 1 \in G/K'}}
r \cdot \vec{e}(\{g_1, \ldots, g_t\})
= \sum_{r \in K'/K} \vec{e}(\{rg_1, \ldots, rg_t\})
= \vec{e}((g_1K'/K) \sqcup \cdots \sqcup (g_t K'/K)),
\]
where we have also used the fact $K'/K$ is a normal subgroup
of $G/K$ to write
$(K'/K)g_i = g_i K'/K$. For any $1 \ne s \in G/K'$, consider the term
\[
b =
p^{\text{edit}}(\vec{v}(p) - \vec{e}(T_1(p)) + s \cdot (\vec{e}(T_1(p)))).
\]
Because $s$ does not stabilize $T_1(p)$, we have
$(s \cdot T_1(p)) \cap S_0(p) \ne \varnothing$. Thus,
$\abs{S_0(b)} < m_0$ (a nonzero number
of elements $f \in G/K$ with multiplicity $0$ must
increase in multiplicity) and these terms can be decomposed
in the desired way by the inductive hypothesis.

We now establish the following inductive hypothesis.
\begin{enumerate}
\item[(1)] Suppose there exists an integer $m_1 > 0$ such that
any irreducible polynomial of the form
$p = ((H_j, f_j)_{j \in J})_K$ with $\abs{S_0(p)} = m_0$
and $\abs{S_1(p)} < m_1$
can be expressed as a combination of generators, lower-degree
terms, and polynomials of the form $p^{\text{edit}}(\vec{v})$
where $\sum_{f \in G/K} m_f(\vec{v}) \leq \binom{\abs{G/K}}{2} + 1$.
\end{enumerate}
We prove the inductive step for $\abs{S_1(p)} = m_1$
by induction on $\abs{S_2(p)}$. At the base case,
suppose $\abs{S_2(p)} = 0$. If $T_2(p) = \varnothing$,
then by Lemma \ref{lem:subclaim},
$\sum_{f \in G/K} m_f(\vec{v}) \leq \binom{\abs{G/K}}{2} +1 $
and we are done. Otherwise, we have $G/K = T_2(p) \sqcup S_1(p)
\sqcup S_0(p)$, where $T_2(p)$ is nonempty.

As before, let $K'/K$ be the stabilizer of $T_2(p)$ under the
action of $G/K$, so that $T_2$ is the union of cosets
\[
T_2(p) = (g_1K'/K) \sqcup \cdots \sqcup (g_tK'/K).
\]
By assumption, $K'/K \ne G/K$, since $T_1(p)$ is nonempty and cannot
be all of $G/K$, since $S_0(p)$ and $S_1(p)$ are both nonempty.
Lemma \ref{lem:multiplication_special} implies
\begin{align*}
p &= (p/\vec{e}(T_2(p))) \cdot
p_{K'}(\vec{e}(\{g_1, \ldots, g_t\})) \\
&\quad -
\sum_{1 \ne s \in G/K'} p^{\text{edit}}\left(\vec{v}(p) - \vec{e}(T_2(p))
+ \sum_{\substack{r \in G/K \\ r \mapsto s}}
r \cdot \vec{e}(\{g_1, \ldots, g_t\}) \right) \\
&= (p/\vec{e}(T_2(p))) \cdot
p_{K'}(\vec{e}(\{g_1, \ldots, g_t\})) \\
&\quad -
\sum_{1 \ne s \in G/K'} p^{\text{edit}}\left(\vec{v}(p) - \vec{e}(T_2(p))
+ s \cdot \vec{e}(T_2(p)) \right)
\end{align*}
$p/\vec{e}(T_2(p))$ is a lower degree term while
$p_{K'}(\vec{e}(\{g_1, \ldots, g_t\}))$ is a type (b) generator
for the same reasons as before.
Consider now the term
\[
b = p^{\text{edit}}(\vec{v}(p) - \vec{e}(T_2(p)) + s \cdot
\vec{e}(T_2(p)))
\]
for $1 \ne s \in G/K'$. Because $s$ does not stabilize $T_2(p)$,
we have $(s \cdot T_2(p)) \cap (S_1(p) \sqcup S_0(p)) \ne \varnothing$.
If $(s \cdot T_2(p)) \cap S_0(p) \ne \varnothing$,
then $\abs{S_0(b)} < m_0$ as before and we are done by the
inductive hypothesis. Otherwise, we must have
$(s \cdot T_2(p)) \cap S_0(p) = \varnothing$ and
$(s \cdot T_2(p)) \cap S_1(p) \ne \varnothing$, so that
$\abs{S_0(b)} = m_0$ and $\abs{S_0(b)} < m_1$, whence we are done
by the inductive hypothesis as well.

We continue the induction in this way; the base case for
step ($k$) follows from the inductive hypotheses
for steps ($j$) with $j < k$. If we can ever prove the inductive
step for any ($k$), then the entire induction completes and we are
done. The inductive step at step ($k$) is to show that
any irreducible polynomial of the form
$p = ((H_j, f_j)_{j \in J})$ with $\abs{S_i(p)} = m_i$
and $m_i > 0$
for all $0 \leq i \leq k$ can be decomposed in the desired
manner. But when $k = \abs{G/K} - 1$,
the fact that $\abs{S_i(p)} \ne \varnothing$ for all
$i \leq \abs{G/K} - 1$ implies by the proof of
Lemma \ref{lem:subclaim} that $T_{k+1}(p) = \varnothing$ and thus
$\sum_{f \in G/K} m_f(p) \leq \binom{\abs{G/K}}{2} + 1$,
so $p$ must already be of the desired form. This completes the proof
of the inductive step and thus of the entire statement.
\end{proof}

\section{Extension to Non-Transitive Sets}\label{sec:boxproduct}

\subsection{A General Formula for the Box Product}

Our goal for this section is to show that the levelwise
grading and finite generation results can be extended to
$\A[X]$ when $X$ is not transitive---in particular, that
transitively finite groups are Tambara finite.
We review the construction of the box product and its
relation with Dress pairings.
Let $G$ be a finite group.
The category of finite $G$-sets is symmetric monoidal with
respect to the product $A \times B$ of two $G$-sets---this
induces a symmetric monoidal product on the categories
$\bispan_G$ and $\bispan_G^+$. The \textit{box product}
of two Mackey functors $M, N$ is defined to be the Day convolution
of $M$ and $N$ with respect to this monoidal structure, i.e.
$M \boxtimes N$ is the left Kan extension
\[
\begin{tikzcd}
\bispan^+_G \times \bispan^+_G \ar[rr, "M \times N"]
\ar[dr, "\otimes", swap]
& \ar[d, phantom, "\Downarrow" description]& \cat{Set} \\
& \bispan^+_G \ar[ur, "M \boxtimes N", swap] &
\end{tikzcd}
\]
More concretely, there exist for any finite $G$-sets $X, Y$ maps
\[
M(X) \times N(Y) \to (M \boxtimes N)(X \times Y)
\]
natural in $X$ and $Y$, and moreover $M \boxtimes N$ is initial
with respect to this property. The universal property of
$M \boxtimes N$ can also be described using Dress pairings:

\begin{lem}[\cite{lewis1981}]
A morphism $M \boxtimes N \to P$ is equivalent to the following
data: for each subgroup $H \leq G$, a bilinear map
$f_H: M(G/H) \otimes N(G/H) \to P(G/H)$ such that the following
are satisfied:
\begin{enumerate}
\item $f_H \circ (\Res_H^K \otimes \Res_H^K) = \Res_H^K \circ f_K$
whenever $H \leq K$.

\item $f_H \circ (c_g \otimes c_g) = c_g \circ f_H$ for all
$g \in G$ and $H \leq G$.

\item For any $H \leq K$,
\[
\Tr_H^K \circ f_H \circ (\Res_H^K \otimes \id)
= f_K \circ (\id \otimes \Tr_H^K),
\quad
\Tr_H^K \circ f_H \circ (\id \otimes \Res_H^K)
= f_K \circ (\Tr_H^K \otimes \id).
\]
\end{enumerate}
\end{lem}

The box product of two Tambara functors
is again Tambara functor in a canonical way:

\begin{lem}[\cite{strickland2012tambarafunctors}]
If $T, R$ are Tambara functors,
there is a unique Tambara functor structure for
the Mackey functor
$T \boxtimes R$ such that the Dress pairings
\[
f_H: T(G/H) \otimes R(G/H) \to (T \boxtimes R)(G/H)
\]
are ring maps and
satisfy $f_K( \Nm_H^K(x) \otimes \Nm_H^K(y))
= \Nm_H^K (f_H(x \otimes y))$ whenever
$x \otimes y \in T(G/H) \otimes R(G/H)$ and
$H \leq K$. So defined,
$\boxtimes$ is the coproduct in the category
of Tambara functors.
\end{lem}

A slight generalization of the Dress pairing
conditions can be formulated for an
$n$-fold box product.

\begin{lem}
A morphism
$M_1 \boxtimes \cdots \boxtimes M_N \to P$
is equivalent to the following data:
for each subgroup $H \leq G$, a multilinear
map $f_H:M_1(G/H) \otimes \cdots
\otimes M_N(G/H) \to P(G/H)$ such that the
following are satisfied:
\begin{enumerate}
\item $f_H \circ (\Res_H^K \otimes \cdots
\otimes \Res_H^K) = \Res_H^K \circ f_K$ whenever
$H \leq K$.

\item $f_H \circ (c_g \otimes \cdots \otimes c_g)
= c_g \circ f_H$ for all $g \in G$ and $H \leq G$.

\item For any $i = 1, \ldots, N$ and $H \leq K$,
\begin{align*}
&\Tr_H^K \circ f_H \circ (\Res_H^K
\otimes \cdots \otimes \Res_H^K
\otimes \id_i \otimes \Res_H^K\otimes
\cdots \otimes \Res_H^K) \\
&\quad =
f_K \circ ( \id_1 \otimes \cdots \otimes \id_{i-1}
\otimes \Tr_H^K
\otimes \id_{i+1} \otimes \cdots \otimes \id_N ),
\end{align*}
where $\id_i$ means that on the $i$\textsuperscript{th}
factor, $\id$ is being applied.
\end{enumerate}
\end{lem}

\begin{proof}[Proof Sketch.]
Given natural morphisms
\[
M_1(X_1) \times \cdots \times M_N(X_N) \to P(X_1 \times
\cdots \times X_N),
\]
the multilinear map $f_H$ is defined via the composition
\[
M_1(G/H) \times \cdots \times M_N(G/H) \to P(G/H \times
\cdots \times G/H) \xrightarrow{R_\delta} P(G/H),
\]
where $R_\delta$ is restriction along the diagonal map
$\delta: G/H \to G/H \times \cdots \times G/H$.

Conversely, given multilinear maps
$f_H: M_1(G/H) \otimes \cdots \otimes M_N(G/H) \to P(G/H)$,
we extend these to multilinear maps
$f_X: M_1(X) \times \cdots \times M_N(X) \to P(X)$
for a nontransitive $G$-set $X = G/H_1 \sqcup \cdots \sqcup G/H_m$
via the following procedure.
Using the fact that
\[
M_j\left( \coprod_i G/H_i \right) \cong \prod_i M_j(G/H_i),
\]
write $x_{i, j}$ for the $M_j(G/H_i)$-component of
an element $(x_1, \ldots, x_N) \in M_1(X) \times \cdots \times M_N(X)$.
Setting
\[
f_X(x_1, \ldots, x_N) =
(f_{H_1}(x_{1, 1}, \ldots, x_{1, N}), \ldots, f_{H_m}(x_{m, 1},
\ldots, x_{m, N})),
\]
the maps $M_1(X_1) \times \cdots \times
M_N(X_N) \to P(X_1 \times \cdots \times X_N)$ are given by
\[
M_1(X_1) \times \cdots \times
M_N(X_N)
\xrightarrow{R_{\pr_i}} \prod_{i=1}^N M_i
\left( \prod_{j=1}^N X_j \right)
\xrightarrow{f_{\prod_{j=1}^N X_j}}
P\left( \prod_{j=1}^N X_j \right).
\]
The maps $f_X$ do not depend on the choice of isomorphism
$X \cong G/H_1 \sqcup \cdots \sqcup G/H_m$, since such a choice
induces an action of $c_{g_i}$ on each $M_j(G/H_i)$ for some
$g_i \in G$, and we know that the $f_{H_i}$ are equivariant
under such an action by assumption.
\end{proof}

We describe a formula for the $n$-fold box product
of two Mackey functors over an arbitrary finite
group $G$. First, a lemma.

\begin{lem}\label{lem:double_coset}
Let $G$ be a group
and let $H, K, L, M$ be subgroups, with
$H \leq K \leq L$ and $M \leq L$.
Then the double coset representatives of
$H \backslash L / M$ are in bijection with the
elements $\{y_x \cdot x\}$, where the $x$ range over the
double coset representatives of $K \backslash L / M$
and the $y_x$, for each fixed $x$, range over the
double coset representatives of $H \backslash K/
(xMx^{-1} \cap K)$.
\end{lem}

\begin{proof}
Suppose $y_x x$ and $y_{x'} x'$ represent the
same double coset $H \backslash L / M$.
Then $y_{x'} x' = h y_x x m$ for $h \in H$ and
$m \in M$, which implies $x$ and $x'$ represent the
same $K \backslash L / M$ coset, so $x = x'$.
Now if $y_x' x = h y_x x m$ for $h \in H$ and
$m \in M$, then
$y_x' = h y_x (xmx^{-1})$; since $y_x, y_x' \in K$,
$xmx^{-1} \in xMx^{-1} \cap K$ and thus
$y_x', y_x$ represent the same
$H \backslash K/ (xMx^{-1} \cap K)$ coset.
This shows injectivity of the map
$\{y_x x\} \to H \backslash L / M$.

For surjectivity, note that
for fixed $x$, changing the representative
of $y_x$ does not affect the double coset $Hy_x x M$.
Indeed, if $y_x' = h y_x \kappa$, then $\kappa = xmx^{-1}$
and $Hy_x' xM= H y_x x M$.

Thus, what we have shown is this: the set of double cosets
$\{Hy_x x M\}$ where $x$ varies over $K \backslash L/M$
and $y_x$ varies over $H \backslash K / (xMx^{-1} \cap K)$ is equal to the
set of double cosets $\{H y x M\}$ where $y$ varies over
$K$ and $x$ varies over $K \backslash L /M$. This is
just the set of double cosets $\{H yx M\}$ where
$x, y$ vary arbitrarily in $L, K$, so we obtain
surjectivity.
\end{proof}

\begin{thm}\label{thm:boxproduct}
Let $G$ be a finite group,
$M_1, \ldots, M_N$ be a collection of $G$-Mackey functors
and write $M = M_1 \boxtimes \cdots \boxtimes M_N$. Fix a subgroup
$L \leq G$. For any subgroup $H \leq L$, define
\[
S_H^L = M_1(G/H) \otimes \cdots \otimes M_N(G/H).
\]
Then we have
\[
M(G/L) = \left( \bigoplus_{H \leq L} S_H^L \right)/F,
\]
where $F$ is the submodule generated by the Frobenius relations
\begin{align*}
&S_K^L \ni
x_1 \otimes \cdots \otimes x_{i-1} \otimes \Tr_H^K(x_i) \otimes
x_{i+1} \otimes \cdots \otimes x_N \\
&\quad = 
\Res_H^K(x_1) \otimes \cdots \otimes \Res_H^K(x_{i-1})
\otimes x_i \otimes
\Res_H^K(x_{i+1}) \otimes \cdots \otimes \Res_H^K(x_N) \in S_H^L
\end{align*}
for all $H \leq K \leq L$,
and the Weyl relations
\[
S_H^L \ni x_1 \otimes \cdots \otimes x_N =
c_\ell(x_1) \otimes \cdots \otimes c_\ell(x_N) \in
S_{\ell H\ell^{-1}}^L
\]
for all $\ell \in L$.

The isomorphisms $c_g: M(G/L) \to M(G/gLg^{-1})$
are induced on each $S_H^L$
by the maps
\[
S_H^L \ni x_1 \otimes \cdots \otimes x_N
\mapsto c_g(x_1) \otimes \cdots \otimes c_g(x_N) \in
S_{gHg^{-1}}^{gLg^{-1}}.
\]
The transfer maps $\Tr_L^{L'}: M(G/L) \to M(G/L')$
are induced by the obvious
isomorphisms $S_H^L \to S_H^{L'}$. The restriction maps
$\Res_L^{L'}: M(G/L') \to M(G/L)$ are defined on each
component $S_H^{L'}$ by the formula
\[
\Res_L^{L'}(x_1 \otimes \cdots \otimes x_N)
= \sum_{g \in L \backslash L'/ H}
c_g \left(
\Res_{H \cap g^{-1} L g}^H(x_1)\right) \otimes \cdots
\otimes c_g \left(\Res_{H \cap g^{-1} L g}^H(x_N) \right),
\]
where each $c_g\left( \Res_{H \cap g^{-1} L g}^H(x_1)\right)
\otimes \cdots
\otimes c_g\left(\Res_{H \cap g^{-1} L g}^H(x_N)\right)$
lies in $S_{gHg^{-1} \cap L}^L$.
The Dress pairing $f_L: M_1(G/L) \otimes \cdots \otimes M_N(G/L)
\to M(G/L)$ is induced by the obvious map
\[
M_1(G/L) \otimes \cdots \otimes M_N(G/L) \to S_L^L.
\]
\end{thm}

\begin{rem}
While the $S_H^L$ depend only on $H$, not $L$,
we write $S_H^L$ as a reminder that these groups
should be understood as formal expressions of the form
$\Tr_H^L(x_1 \otimes \cdots \otimes x_N)$,
where each $x_i$ is in $M_i(G/H)$.
\end{rem}

\begin{proof}
For notational simplicity, we give the proof when $N = 2$,
but the proof for arbitrary $N$ is exactly the same.
There is a long chain of straightforward
but necessary verifications
which need to be performed.

\begin{enumerate}
\item[(1)] $\Tr$, $\Res$, and $c_g$ are
well-defined.
\end{enumerate}

This is obvious for $\Tr_L^{L'}$ and straightforward for $c_g$.
Observe that $\Tr_L^{L'}$ and $\Res_L^{L'}$ preserve the
Weyl relations, since $\Res$ and $\Tr$ commute with $c_g$
for both $M_1$ and $M_2$.
Moreover, the definition of $\Res_L^{L'}$ does not
depend on the choice of double coset representatives, since
if $g' = \ell g h$ for $\ell \in L$ and $h \in H$, then
\begin{align*}
&c_{g'} \left( \Res_{H \cap {g'}^{-1} L g'}^H(x_1) \right) \otimes
c_{g'} \left( \Res_{H \cap {g'}^{-1} L g'}^H(x_2) \right) \\
&= c_\ell c_g c_h \left( \Res_{H \cap h^{-1}g^{-1} L gh}^H(x_1) \right)
\otimes c_\ell c_g c_h
\left( \Res_{H \cap h^{-1}g^{-1} L gh}^H(x_2) \right)\\
&= c_\ell c_g \left( \Res_{H \cap g^{-1} L g}^H(c_h(x_1)) \right)
\otimes c_\ell c_g
\left( \Res_{H \cap g^{-1} L g}^H(c_h(x_2)) \right)\\
&= c_g \left( \Res_{H \cap g^{-1} L g}^H(x_1) \right)
\otimes c_g
\left( \Res_{H \cap g^{-1} L g}^H(x_2) \right),
\end{align*}
where we have used the Weyl relations in the last line
and the fact that $c_h(x_i) = x_i$ for $i = 1, 2$.

The only nontrivial part is proving that $\Res_L^{L'}$ preserves
Frobenius relations. Let $H \leq K \leq L'$.
By definition,
\[
\Res_L^{L'}(x \otimes \Tr_H^K(y))
= \sum_{g \in L \backslash L' / K}
c_g \left( \Res_{K \cap g^{-1} L g }^K (x)\right) \otimes
c_g\left(\Res_{K \cap g^{-1} L g }^K \Tr_H^K (y) \right).
\]
On the other hand,
\[
\Res_L^{L'}(\Res_H^K(x) \otimes y)
= \sum_{g \in L \backslash L'/ H} c_g
\left(\Res_{H \cap g^{-1}Lg}^K(x) \right)
\otimes c_g\left( \Res_{H \cap g^{-1}Lg}^H(y) \right).
\]
By Lemma \ref{lem:double_coset}, this is equal to
\begingroup
\allowdisplaybreaks
\begin{align*}
&\sum_{s \in L \backslash L' / K}
\sum_{t \in (s^{-1}Ls \cap K) \backslash K / H}
c_sc_t
\left(\Res_{H \cap (st)^{-1}L(st)}^K(x) \right)
\otimes c_sc_t \left( \Res_{H \cap (st)^{-1}L(st)}^H(y) \right) \\
&= \sum_{s \in L \backslash L' / K}
\sum_{t \in (s^{-1}Ls \cap K) \backslash K / H}
c_s
\left(\Res_{tHt^{-1} \cap s^{-1}Ls}^K (x) \right)
\otimes c_s \left( c_t \Res_{H \cap (st)^{-1}L(st)}^H(y) \right) \\
&= \sum_{s \in L \backslash L' / K}
\sum_{t \in (s^{-1}Ls \cap K) \backslash K / H}
c_s
\left(\Res_{K \cap s^{-1}Ls}^K (x) \right)
\otimes c_s \left( \Tr_{tHt^{-1} \cap s^{-1}Ls}^{K \cap s^{-1}Ls}
c_t \Res_{H \cap (st)^{-1}L(st)}^H(y) \right) \\
&= \sum_{s \in L \backslash L' / K}
c_s \left( \Res_{K \cap s^{-1} L s }^K (x)\right) \otimes
c_s\left(\Res_{K \cap s^{-1} L s }^K \Tr_H^K (y) \right)
\end{align*}
\endgroup
which by the above is equal to $\Res_L^{L'}(x \otimes \Tr_H^K(y))$.
On the second-to-last line, we have used the fact that
\begin{align*}
&\Res_{tHt^{-1} \cap s^{-1}Ls}^K (x)
\otimes c_t \Res_{H \cap (st)^{-1}L(st)}^H(y) \\
&= \Res_{K \cap s^{-1}Ls}^K \Res_{tHt^-1 \cap s^{-1}Ls}^{K \cap s^{-1}Ls}
(x) \otimes c_t \Res_{H \cap (st)^{-1} L (st)}^H (y) \\
&= \Res_{tHt^-1 \cap s^{-1}Ls}^{K \cap s^{-1}Ls}
(x) \otimes \Tr_{tHt^{-1} \cap s^{-1}Ls}^{K \cap s^{-1}Ls}
c_t \Res_{H \cap (st)^{-1} L (st)}^H (y)
\end{align*}
by the Frobenius relations. On the last line, we have used the
fact that
\[
\Res_{K \cap s^{-1} L s}^K \Tr_H^K
= \sum_{t \in (s^{-1} L s \cap K) \backslash K / H}
\Tr_{tHt^{-1} \cap s^{-1}Ls}^{K \cap s^{-1}Ls}
c_t \Res_{H \cap (st)^{-1} L (st)}^H
\]
by the double coset formula.

\begin{enumerate}
\item[(2)] $\Tr$, $\Res$, $c_g$ are functorial.
\end{enumerate}

This is again more or less obvious by definition for
$\Tr$ and $c_g$. For $\Res$, let $L \leq L' \leq L''$,
$H \leq L''$, and suppose $x_1 \otimes \cdots
\otimes x_N \in S_H^{L''}$. By definition,
\begingroup
\allowdisplaybreaks
\begin{align*}
&\Res_L^{L'} \Res_{L'}^{L''} (x \otimes y) \\
&= \Res_L^{L'} \left(
\sum_{g \in L' \backslash L'' / H}
c_g \left( \Res_{H \cap g^{-1} L' g}^H (x) \right) \otimes
c_g \left(\Res_{H \cap g^{-1} L' g}^H(y) \right) \right) \\
&= \sum_{g \in L' \backslash L'' / H}
\sum_{t \in L \backslash L' / (gHg^{-1} \cap L')} c_t \Big(\Res_{t^{-1}Lt \cap gHg^{-1}}^{L' \cap gHg^{-1}} c_g
\Res_{H \cap g^{-1} L' g}^H (x)
\\
&\hspace{6cm}
\otimes \Res_{t^{-1}Lt \cap gHg^{-1}}^{L' \cap gHg^{-1}} c_g
\Res_{H \cap g^{-1} L' g}^H(y) \Big) \\
&= \sum_{g \in L' \backslash L'' / H}
\sum_{t \in L \backslash L' / (gHg^{-1} \cap L')}
c_{tg} \left( \Res_{H \cap (tg)^{-1} L (tg)}^H(x)
\otimes \Res_{H \cap (tg)^{-1} L (tg)}^{H}(y)\right) \\
&= \sum_{g \in L \backslash L'' / H}
c_g \left( \Res_{H \cap g^{-1} L g}^H (x) \otimes
\Res_{H \cap g^{-1} L g}^H (y) \right)
= \Res_L^{L''}(x \otimes y),
\end{align*}
\endgroup
where we have used Lemma \ref{lem:double_coset} in the last line.

\begin{enumerate}
\item[(3)] $M$ assembles into a Mackey functor
via $\Tr, \Res, c_g$.
\end{enumerate}

It is clear that $c_g$ commutes with $\Tr$.
furthermore, $c_\ell = \id$ on $M(G/L)$ by definition.
For commutation with $\Res$, let $x_1 \otimes \cdots
\otimes x_N \in S_H^{L'}$; we have
\[
c_g \Res_L^{L'}(x_1 \otimes \cdots \otimes x_N)
= \sum_{t \in L \backslash L'/ H} c_gc_t
\left( \Res_{H \cap t^{-1}Lt}^H(x_1)
\otimes \cdots \otimes \Res_{H \cap t^{-1}Lt}^H(x_N) \right),
\]
while on the other hand
\begin{align*}
&\Res_{gLg^{-1}}^{gL'g^{-1}}
(c_g(x_1) \otimes \cdots \otimes c_g(x_N)) \\
&= \sum_{t \in L \backslash L'/H}
c_{gtg^{-1}}
\left( \Res_{gHg^{-1} \cap gt^{-1}Ltg^{-1}}^{gHg^{-1}}
(c_g(x_1)) \otimes \cdots \otimes
\Res_{gHg^{-1} \cap gt^{-1}Ltg^{-1}}^{gHg^{-1}} (c_g(x_N))
\right) \\
&= \sum_{t \in L \backslash L' / H}
c_g c_t
\left( \Res_{H \cap t^{-1}Lt}^H(x_1)
\otimes \cdots \otimes \Res_{H \cap t^{-1}Lt}^H(x_N) \right)
\end{align*}
as desired.
Thus, it remains to verify the double coset formula.
Let $L, L' \leq L''$ and
$x_1 \otimes \cdots \otimes x_N \in S_H^{L'}$.
\begin{align*}
\Res_L^{L''} \Tr_{L'}^{L''}(x_1 \otimes \cdots \otimes x_N)
&= \sum_{g \in L \backslash L'' / H}
c_g \left( \Res_{H \cap g^{-1} L g}^H(x_1)
\otimes \cdots \otimes \Res_{H \cap g^{-1} L g}^H (x_N) \right),
\end{align*}
while on the other hand
\begingroup
\allowdisplaybreaks
\begin{align*}
&\sum_{g \in L \backslash L'' / L'}
\Tr_{L \cap gL'g^{-1}}^L c_g \Res_{g^{-1}Lg \cap L'}^{L'}
(x_1 \otimes \cdots \otimes x_N) \\
&= \sum_{g \in L \backslash L''/L'}
\sum_{t \in (g^{-1}Lg \cap L') \setminus L' / H}
c_g c_t \left(
\Res_{(gt)^{-1} L (gt) \cap H}^H(x_1) \otimes \cdots
\otimes \Res_{(gt)^{-1} L (gt) \cap H}^H(x_N) \right) \\
&= \sum_{g \in L \backslash L''/ H}
c_g \left( \Res_{g^{-1}Lg \cap H}^H (x_1)
\otimes \cdots \otimes \Res_{g^{-1}Lg \cap H}^H (x_N) \right),
\end{align*}
\endgroup
using Lemma \ref{lem:double_coset} again. Thus,
$M$ assembles into a Mackey functor.

\begin{enumerate}
\item[(4)] $M \cong M_1 \boxtimes \cdots \boxtimes M_N$.
\end{enumerate}

We verify that $M$ has the universal property of
$M_1 \boxtimes \cdots \boxtimes M_N$ with respect to the proposed
pairing $f_L: M_1(G/L) \otimes \cdots \otimes M_N(G/L) \to M(G/L)$.
The fact that the $f_L$ do indeed assemble into a Dress pairing
is enforced by the Frobenius and Weyl
relations. Given any Dress pairing
$g_L: M_1(G/L) \otimes \cdots \otimes M_N(G/L) \to P(G/L)$,
we define $h_L: M(G/L) \to P(G/L)$ by setting
\[
h_L(x_1 \otimes \cdots \otimes x_N)
= \Tr_H^L( g_H(x_1 \otimes \cdots \otimes x_N))
\]
whenever $x_1 \otimes \cdots \otimes x_N \in S_H^L$.
Conversely, given a morphism of Mackey functors $h: M \to P$,
the Dress pairing $g_L: M_1(G/L) \otimes \cdots \otimes M_N(G/L)
\to P(G/L)$ is recovered via the composition
$g_L = h_L \circ f_L$, where $h_L: M(G/L) \to P(G/L)$ is the
$L$-level of $h$.
\end{proof}

The formula simplifies slightly in the abelian case:

\begin{cor}
Let $G$ be a finite abelian group,
$M_1, \ldots, M_N$ be a collection of $G$-Mackey functors
and write $M = M_1 \boxtimes \cdots \boxtimes M_N$. Fix a subgroup
$L \leq G$. For any subgroup $H \leq L$, define
\[
S_H^L = (M_1(G/H) \otimes \cdots \otimes M_N(G/H))/(L/H).
\]
Then we have
\[
M(G/L) = \left( \bigoplus_{H \leq L} S_H^L \right)/F,
\]
where $F$ is the submodule generated by the Frobenius relations
\begin{align*}
&S_K^L \ni
x_1 \otimes \cdots \otimes x_{i-1} \otimes \Tr_H^K(x_i) \otimes
x_{i+1} \otimes \cdots \otimes x_N \\
&\quad = 
\Res_H^K(x_1) \otimes \cdots \otimes \Res_H^K(x_{i-1})
\otimes x_i \otimes
\Res_H^K(x_{i+1}) \otimes \cdots \otimes \Res_H^K(x_N) \in S_H^L
\end{align*}
For all $H \leq K \leq L$.
The transfer, restriction, and conjugation homomorphisms
are defined analogously.
\end{cor}

\begin{rem}
The formula of Theorem \ref{thm:boxproduct} generalizes the
inductive description of $T \boxtimes R$ for two
$C_{p^n}$-Tambara functors given in
\cite{MR3187614} to all finite $G$.
\end{rem}

\begin{rem}
When $T_1, \ldots, T_N$ are Tambara functors, the universal
morphisms $T_i \to T = T_1 \boxtimes \cdots \boxtimes T_N$
exhibiting $T$ as a coproduct of the $T_i$
are given by the compositions
\[
T_i(G/L) \xrightarrow{g_i} T_1(G/L)
\otimes \cdots \otimes T_N(G/L) \xrightarrow{f_L} T(G/L),
\]
where $g_i: T_i(G/L) \to T_1(G/L)
\otimes \cdots \otimes T_N(G/L)$
is the map sending $x$ to $1 \otimes \cdots \otimes 1 \otimes
x \otimes 1 \otimes \cdots \otimes 1$, i.e. inclusion into the
$i$\textsuperscript{th} factor. In the special case that
$T_i = \A[G/H_i]$ and $T = \boxtimes_{i=1}^N \A[G/H_i]
= \A[X]$, the universal morphisms $T_i \to T$ are induced
by $R_{g_i} \in \bispan_G(X, G/H_i)$, where
$g_i: G/H_i \to X$ is the canonical inclusion (follows from
an application of the Yoneda lemma).
\end{rem}

\begin{rem}
Let $T$ be a Tambara functor and let $R, S$ be $T$-algebras.
The box product $R \boxtimes_T S$ is given by modifying
the definition of each $S_H^L$ to be tensor products over
$T(G/H)$.
\end{rem}

\subsection{Extending Finiteness Criteria to
Non-Transitive Sets}

\begin{prop}\label{prop:boxproductfinite}
Let $G$ be a finite group.
If $T, R$ are levelwise finitely generated
$G$-Tambara functors, then so is $T \boxtimes R$.
\end{prop}

\begin{proof}
Fix an $L \leq G$. We claim that the image of
each $S_H^L$ in $(T \boxtimes R)(G/L)$ is a finitely
generated $S_L^L$-module. Since
$S_L^L \cong T(G/L) \otimes R(G/L)$ as a ring,
which is finitely generated, this will imply
$(T \boxtimes R)(G/L)$ is finitely generated.
Indeed, the elements of $S_H^L$ are all of the form
$\Tr_H^L(x \otimes y)$ where $x \otimes y \in S_H^H$.
$S_L^L$ acts by
\[
(x' \otimes y') \cdot \Tr_H^L(x \otimes y)
= \Tr_H^L (\Res_H^L(x' \otimes y') \cdot (x \otimes y))
= \Tr_H^L ( \Res_H^L(x')x \otimes \Res_H^L(y') y ).
\]
In other words, the image of $S_H^L$ is a quotient
of the $S_L^L$-module
$T(G/H) \otimes R(G/H)$, with $S_L^L$-action given by
the restriction map
\[
\Res_H^L \otimes \Res_H^L:
T(G/L) \otimes T(G/L) \to T(G/H) \otimes R(G/H).
\]
Since $T, R$ are levelwise finitely generated
and the maps $\Res_H^L: T(G/L) \to T(G/H)$ are
integral (Lemma \ref{lem:integral}), they are finite ring maps, whence
$T(G/L) \otimes T(G/L) \to T(G/H) \otimes R(G/H)$
is a finite ring map and thus $T(G/H) \otimes R(G/H)$
is a finite $S_L^L$-module, from which the result follows.
\end{proof}

\begin{rem}
The same proof can be extended to show that
if $S$ is a $G$-Tambara functor and $T, R$ are $S$-algebras,
levelwise finitely generated over $S$, then
so is $T \boxtimes_S R$.
\end{rem}

\begin{cor}\label{cor:transitively_implies_Tambara}
Let $G$ be a transitively Tambara finite group. Then $G$ is also
Tambara finite.
\end{cor}

\begin{proof}
This follows from Proposition
\ref{prop:boxproductfinite}
and the isomorphism $\A[X \sqcup Y] \cong \A[X]
\boxtimes \A[Y]$.
\end{proof}

\begin{ex}\label{ex:green_example}
We remark that Proposition \ref{prop:boxproductfinite}
is not immediate from the formula in Theorem $\ref{thm:boxproduct}$,
since it is is false for Green functors.
Let $R$ be the $C_p$-Green functor given by the Lewis diagram
\[
\begin{tikzcd}
\F_p \ar[d, xshift=-3pt, "\Res_e^{C_p}", swap] \\
\F_p[x] \ar[u, xshift=3pt, swap, "\Tr_e^{C_p}"]
\ar[in=290, out=250, loop, looseness=4, "\id", swap]
\end{tikzcd}
\]
where $\Res_e^{C_p}: \F_p \to \F_p[x]$ is the standard
inclusion and $\Tr_e^{C_p} = 0$. Clearly $R$ is a levelwise
finite-type $\Z$-algebra.
$R \boxtimes R$ is given by
\[
\begin{tikzcd}
(\F_p \ar[d, xshift=-3pt, "\Res_e^{C_p}", swap] \oplus
\F_p[x, y])/F\\
\F_p[x, y] \ar[u, xshift=3pt, swap, "\Tr_e^{C_p}"]
\ar[in=290, out=250, loop, looseness=4, "\id", swap]
\end{tikzcd}
\]
The Frobenius relations enforce
$0 = c \otimes \Tr(y^k) = \Tr(c \otimes y^k)$ and
$0 = \Tr(x^k) \otimes c = \Tr(x^k \otimes c)$
for $k \geq 0$,
and multiplication
\[
\Tr(x^{n_1} \otimes y^{m_1})
\Tr(x^{n_2} \otimes y^{m_2})
= \Tr \left( (x^{n_1} \otimes y^{m_1})
\Res \Tr(x^{n_2} \otimes y^{m_2}) \right)
= 0
\]
Hence the top level $(R \boxtimes R)(C_p/C_p)$ is the ring
consisting of elements
\[
(n, p(x, y))
\]
where $n \in \F_p$ and $p(x, y)$ is a polynomial with
$\F_p$-coefficients having no terms of the form
$cx^k, cy^k$, $k \geq 0$. Multiplication is given by
$(n, p(x, y))(m, q(x, y)) = (nm, nq(x, y) + mp(x, y))$.
In particular, the top level $(R \boxtimes R)(C_p/C_p)$ is not
a finitely generated ring.
\end{ex}

While
levelwise finite generation is too much to ask
for a box product of levelwise finitely generated
Green functors (see Example \ref{ex:green_example}),
relative finite-dimensionality is not:

\begin{lem}\label{lem:relatively_finite_green}
Let $T, R$ be relatively finite dimensional Green functors.
Then $T \boxtimes R$ is also relatively
finite-dimensional.
\end{lem}

\begin{proof}
Let $L \leq L'$.
As in the proof of Proposition \ref{prop:boxproductfinite},
each $S_H^L$ is a finite module over $S_L^L =
T(G/L) \otimes R(G/L)$, since the $S_L^L$-module structure
on $S_H^L = T(G/H) \otimes R(G/H)$ is induced by
$\Res_H^L \otimes \Res_H^L: T(G/L) \otimes R(G/L)
\to T(G/H) \otimes R(G/H)$ which is a finite ring map
by assumption. Thus,
$(T \boxtimes R)(G/L)$ is a finite
$T(G/L) \otimes R(G/L)$-module. Since
$\Res_L^{L'}: T(G/L') \otimes R(G/L')
\to T(G/L) \otimes R(G/L)$ is a finite ring map,
we therefore see that $\Res_L^{L'}: T(G/L') \otimes R(G/L')
\to (T \boxtimes R)(G/L)$ is a finite ring map,
and hence $\Res_L^{L'}:
(T \boxtimes R)(G/L') \to (T \boxtimes R)(G/L)$ is as well.
\end{proof}

We also remark that with relative finite-dimensionality
hypotheses, ``levelwise finite generation'' is stable
under base change:

\begin{lem}
Let $T$ be a relatively finite-dimensional Tambara functor,
and suppose $R$ is levelwise finitely generated (over $\A$). Then
$T \boxtimes R$ is levelwise finitely generated over $T$.
\end{lem}

\begin{proof}
Let $L \leq G$. For each $H \leq L$, we have that
the ring map $\Res_H^L \otimes \Res_H^L:
T(G/L) \otimes R(G/L) \to T(G/H) \otimes R(G/H)$ is finite,
from which we deduce (as in the proof of Proposition
\ref{prop:boxproductfinite}) that the ring map
$T(G/L) \otimes R(G/L) \to (T \boxtimes R)(G/L)$
is finite. Since $R(G/L)$ is a finite-type $\Z$-algebra,
$T(G/L) \otimes R(G/L)$ is a finite-type $T(G/L)$-algebra,
hence $(T \boxtimes R)(G/L)$ is a finite-type $T(G/L)$-algebra.
\end{proof}

\begin{thm}\label{thm:relatively_finite_general}
Let $G$ be a finite group and let $X$ be a finite $G$-set.
Then $\A[X]$ is relatively
finite-dimensional.
\end{thm}

\begin{proof}
Follows from Lemma \ref{lem:relatively_finite_green},
Lemma \ref{lem:relatively_finite}, and the fact that
$\A[X \sqcup Y] \cong \A[X] \boxtimes \A[Y]$.
\end{proof}




We showed that all finite Dedekind groups
are transitively Tambara finite in \S\ref{sec:finite_generation}.

\begin{thm}\label{thm:finite_generation_general}
Let $G$ be a finite Dedekind group.
Then $G$ is Tambara finite. In particular, since
$\A[X]$ is a levelwise finitely generated Tambara functor
for any finite $G$-set $X$, it is also levelwise Noetherian
and relatively finite-dimensional.
\end{thm}

\begin{proof}
Follows from Corollary \ref{cor:transitively_implies_Tambara}.
\end{proof}



\begin{thm}
Let $G$ be a Tambara finite and let $X$ be a finite $G$-set.
Then $\A[X]$ is module-Noetherian, in the sense that
every submodule of a finitely generated module over $\A[X]$
is finitely generated.
In particular, this holds for all finite groups
satisfying $(\dagger)$.
\end{thm}

This is a consequence of Theorem
\ref{thm:finite_generation_general}
and a more general fact
about relatively finite-dimensional Green functors:

\begin{lem}[{\cite[Corollary 3.35]{chanwisdom2025algebraicktheory}}]
Let $R$ be a relatively finite-dimensional Green functor.
Then $R$ is module-Noetherian iff $R(G/G)$ is Noetherian
iff $R(G/H)$ is Noetherian for all $H \leq G$.
\end{lem}


Tambara finite groups also satisfy a weak Hilbert basis theorem:

\begin{thm}[Weak Hilbert Basis Theorem]
\label{thm:hilbert_basis}
Let $G$ be a Tambara finite group, and suppose
$T$ is levelwise Noetherian and relatively finite-dimensional.
Then $T[X]$ is also levelwise Noetherian and relatively
finite-dimensional for all finite $G$-sets $X$. If $T$
is levelwise finitely generated, then so is $T[X]$.
\end{thm}

\begin{proof}
$T[X]$ is relatively finite-dimensional by
Lemma \ref{lem:relatively_finite_green}. For the levelwise
Noetherian statement, observe as in the proof of
Proposition \ref{prop:boxproductfinite} that each ring map
$T(G/L) \otimes \A[X](G/L) \to T[X](G/L)$ is finite;
since $T(G/L) \otimes \A[X](G/L)$ is a finite-type
$T(G/L)$-algebra by Theorem \ref{thm:finite_generation_general},
so is $T[X](G/L)$, whence $T[X](G/L)$ is Noetherian.
The levelwise finitely generated part of the statement
follows from Theorem \ref{thm:finite_generation_general} and
Proposition \ref{prop:boxproductfinite}.
\end{proof}

\subsection{Stability of the Grading}

To conclude this section, we provide proofs to
some earlier remarks about the
stability of the grading we have introduced on
polynomials under box products.

\begin{cor}
Let $G$ be a finite group.
The gradings on $\A[G/H]$ for $H \leq G$ induce
a grading on $\A[X]$ for arbitrary finite $G$-sets $X$ via
the box product; for now,
we call this grading the \textit{box grading.}
\end{cor}

\begin{proof}
We explain this for the Dedekind case; the nonabelian
case is obtained by dropping the nonnumerical component
of the degree.
For each $i$, let $T_i = \A[G/H_i]$ and
$T = \A[X] = \boxtimes_{i=1}^N T_i$. We adopt the same
notation for the $S_H^L$ in Theorem \ref{thm:boxproduct}.
Note that each $S_H^L$ has a grading over $\Oh_H^w$ induced
from the grading on each $T_i(G/H)$; by identifying
$\Oh_H^w$ as a subset of $\Oh_L^w$, this induces a grading of
$S_H^L$ over $\Oh_L^w$. We claim that this endows
$T(G/L)$ with the structure of a graded ring over $\Oh_L^w$,
for which it suffices to show that the Frobenius relations
are homogeneous.

Indeed, suppose $H \leq K$ and
$x_i$, $i = 1, \ldots, \hat{j}, \ldots N$ are elements
of degree $(n_i, J_i)$ in $T_i(G/K)$,
while $x_j$ is an element of degree
$(n_j, J_j)$ in $T_j(G/H)$. Then $\Tr_H^K(x_j)$ is also
an element of degree $(n_j, J_j)$ in $T_j(G/K)$ by
Lemma \ref{lem:tr_res_grading}, so
$x_1 \otimes \cdots \otimes \Tr_H^K(x_j) \otimes \cdots \otimes x_N$
is of degree $(\sum n_i, \bigcap J_i)$ in $S_K^L$. On the other hand,
$\Res_H^K(x_1) \otimes \cdots \otimes x_j \otimes \cdots
\otimes \Res_H^K(x_N)$ is of degree $(\sum n_i, \bigcap J_i \cap H)$
in $S_K^L$. Since $J_j \leq H$ by assumption,
these two degrees coincide and $\A[X]$ is graded.
\end{proof}

\begin{rem}
We note that the natural isomorphism
$\A[X] \to \A[X \sqcup \varnothing] \cong \A[X] \otimes \A[\varnothing]
\cong \A[X]$ is graded by inspection
(when $\A[X \sqcup \varnothing]$ is given the box grading),
since all elements of $\A[\varnothing](G/L)$ have degree
$(0, L)$.
Hence the box grading is
well-defined.
\end{rem}

\begin{lem}
The box grading on $\A[X]$ coincides with the naive grading
of Remark \ref{rem:naive}.
\end{lem}

\begin{proof}
Again, we explain this for the Dedekind case; the nonabelian
case is obtained by dropping the nonnumerical component
of the degree.
Write $X = \sqcup_i G/H_i$, $T = \A[X]$, and $T_i = \A[G/H_i]$.
We want to show that the grading coincides on $\A[X](G/L)$.
Since the transfer homomorphisms $\Tr_L^{L'}$ are graded
in the same way for both gradings, it suffices by induction
to show that the grading coincides on elements of
$\A[X](G/L)$ which do not lie in the image of transfer
from lower levels, i.e.
it suffices to show this for elements in the image of the
Dress pairing $f_L: T_1(G/L) \otimes \cdots \otimes T_N(G/L) \to
T(G/L)$. This is a graded ring homomorphism for the box
grading (essentially by definition), so it suffices to show
that the elements $f_L(1 \otimes \cdots \otimes 1 \otimes x_i
\otimes 1 \otimes \cdots \otimes 1)$ have the same degree
with respect to both the box and naive gradings. Now
$f_L(1 \otimes \cdots \otimes 1 \otimes x_i
\otimes 1 \otimes \cdots \otimes 1)$ is the image of
$x_i \in T_i(G/L)$ under the natural map
$T_i \to T$, which is given by
\[
[G/H_i \xleftarrow{f} A \to B \to G/L] \mapsto
[X \xleftarrow{j \circ f} A \to B \to G/L],
\]
where $j: G/H_i \to X$ is the natural inclusion. Thus, if
$x_i$ has degree $(n, K)$ with respect to the naive grading,
then it also has degree $(n, K)$ with respect to the box grading.
\end{proof}

\begin{cor}
Let $f: X \to Y$ be a map of finite $G$-sets.
The induced map $R_f^\ast: \A[X] \to \A[Y]$ is levelwise graded
with respect to the identity homomorphism $\Oh_L^w \to \Oh_L^w$.
If moreover $f$ has a well-defined degree, e.g.
if $Y$ is transitive, then
$N_f^\ast: \A[Y] \to \A[X]$ is levelwise graded with
respect to the monoid homomorphism $\phi: \Oh_L^w \to \Oh_L^w$
given by
\[
\phi(n, K) = (\deg(f) n, K).
\]
\end{cor}

\begin{proof}
$R_f^\ast$ is given by
\[
[X \xleftarrow{p} A \to B \to G/L]
\mapsto
[Y \xleftarrow{f \circ p} A \to B \to G/L],
\]
which evidently preserves the degree of $A \to B$.
$N_f^\ast$ is given by
\[
[Y \ot A \to B \to G/L]
\mapsto
[X \ot A \times_Y X \to B \to G/L],
\]
which changes the degree by a multiplicative factor of
$\deg(f)$.
\end{proof}

\begin{rem}
The corresponding
result for general $G$ is obtained by dropping the
non-numerical component.
\end{rem}

\begin{lem}
Let $X$ be a finite $G$-set and $T$ a Tambara functor.
The free $T$-algebra $T[X] \cong T \boxtimes \A[X]$
inherits a levelwise grading from $\A[X]$ over $\N$.
\end{lem}

\begin{proof}
We again adopt the notation of Theorem \ref{thm:boxproduct}.
Each $T(G/L)$ is trivially graded over $\N$
by giving every element degree $0$. Hence,
each $S_H^L$ is graded over $\N$. It remains
to verify that the Frobenius relations are homogeneous.
But this is obvious by inspection, since $\Res_H^K$
and $\Tr_H^K$ on $\A[X]$ do not alter the numerical component
of the degree.
\end{proof}

\section{Applications to Norm Functors}\label{sec:norms}

Let $G$ be a finite group and suppose $H \leq G$.
The functor $\Ind_H^G$ sending an $H$-set $X$ to its induced
$G$-set $G \times_H X$ is left adjoint to the functor
$\Res_H^G$ which sends a $G$-set to its underlying $H$-set.
This induces an adjunction on the polynomial categories
$\bispan_H$ and $\bispan_G$ (where restriction
is \textit{left} adjoint), which further
induces an adjunction on
the categories $H\cat{Tamb}$ and $G\cat{Tamb}$. More precisely,
defining for any Tambara functor $r_H^G T = T \circ \Ind_H^G$
and $c_H^G T = T \circ \Res_H^G$, we have
\[
\Hom_{H\cat{Tamb}}(r_H^G T, R)
\cong \Hom_{G\cat{Tamb}}(T, c_H^G R).
\]
The functor $r_H^G$ has a left adjoint,
denoted $n_H^G: H\cat{Tamb} \to G\cat{Tamb}$,
which is defined pointwise
by the left Kan extension
\[
\begin{tikzcd}
\bispan_H \ar[rr, "T"] \ar[dr, "\Ind_H^G", swap]
& \ar[d, phantom, "\Downarrow" description]& \cat{Set} \\
& \bispan_G \ar[ur, "n_H^G T", swap] &
\end{tikzcd}
\]
and thus given by the formula
\[
n_H^G T (Y) = \int^{X \in \bispan_H}
\bispan_G(G \times_H X, Y) \times T(X).
\]
In particular, we see from this formula that
$n_H^G$ preserves quotients (levelwise surjections) $T \to R$.
See \cite{rolfthesis} or \cite{hill2019equivariant}
for more details on the functor $n_H^G$.

If $X$ is an $H$-set, we compute
\begin{align*}
\Hom_{G\cat{Tamb}}(n_H^G \A[X], T)
&\cong \Hom_{H\cat{Tamb}} (\A[X], r_H^G T) \\
&\cong (r_H^G T)(X) \\
&\cong T(G \times_H X),
\end{align*}
so there is a natural isomorphism
$n_H^G\A[X] \cong \A[G \times_H X]$.

\begin{lem}
Let $G$ be a finite group. The following are equivalent:
\begin{enumerate}
\item The norm functors $n_H^K$ preserve levelwise
finite generation over
$\Z$ for all finite groups $H \leq K \leq G$.

\item Every subgroup $H\leq G$ is Tambara finite.

\item Every subgroup $H \leq G$ is transitively Tambara finite.
\end{enumerate}
\end{lem}

\begin{proof}
The equivalence of (b) and (c) follows from our work in
\S\ref{sec:boxproduct}. Assume (a); in light of the isomorphism
$n_H^K \A[H/H] \cong \A[K/H]$ and the levelwise
finite generation of
$\A[H/H]$ established in 
\cite{schuchardt2025algebraicallyclosedfieldsequivariant},
$\A[K/H]$ is finitely generated and thus (a) implies (c).
Now assume (b) and let $T$ be a levelwise finitely generated
$H$-Tambara functor. Then $T$ receives a (levelwise) surjection
from $\A[X]$ for some finite $H$-set $X$: the (finite)
collection of generators $x_i$ in the $T(H/H_i)$ for $H_i \leq H$
induce a morphism $\boxtimes_i \A[H/H_i] \to T$ which includes
all the $x_i$ in its image, so $\boxtimes_i \A[H/H_i] \to T$
is a surjection. In our situation, the surjection
$\A[X] \to T$ induces
a surjection
$\A[K \times_H X] \cong n_H^K \A[X] \to n_H^K T$,
so $n_H^K T$ is levelwise finitely generated.
\end{proof}

\begin{cor}\label{cor:norm_preservation_general}
The norm functors $n_H^G$ preserve levelwise
finite generation over $\Z$ for all
finite $H \leq G$ iff all finite groups $G$ are
Tambara finite.
\end{cor}




\begin{thm}\label{thm:norm_preservation}
Let $H \leq G$, where $G$ is a finite Dedekind group.
Then $n_H^G$ preserves levelwise finite generation.
\end{thm}

\printbibliography

\end{document}